\documentclass[10pt]{amsart}

\usepackage{latex_base}
\usepackage{macros}
\title[]{The polyhedral type of a polynomial map on the plane}
\author{Boulos El Hilany, Kemal Rose}
\thanks{For this work, the first author was partially supported by the DFG Walter Benjamin Programme: EL 1092/1-1}
\begin{document}

\maketitle

\begin{abstract} 
Two continuous maps $f, g : \mathbb{C}^2\to\mathbb{C}^2$ are said to be topologically equivalent if there exist homeomorphisms $\varphi,\psi:\mathbb{C}^2\to\mathbb{C}^2$ satisfying $\psi\circ f\circ\varphi = g$. 
It is known that there are finitely many topologically non-equivalent polynomial maps $\mathbb{C}^2\to\mathbb{C}^2$ with any given degree $d$. The number $T(d)$ of these topological types is known only whenever $d=2$. In this paper, we describe the topology of generic complex polynomial maps on the plane using the corresponding pair of Newton polytopes and establish a method for constructing topologically non-equivalent maps of degree $d$. We furthermore provide a software implementation of the resulting algorithm, and present lower bounds on $T(d)$ whenever $d=3$ and $d=4$.
\end{abstract}
 \markleft{}
 \markright{}
%\setcounter{tocdepth}{1}
%\tableofcontents
 
\section{Introduction}\label{sec:intro}
Polynomial maps $f:=(f_1,\ldots,f_n):\CtCmn$ are some of the most classical objects in algebraic geometry and singularity theory. 
Questions regarding their stability, singularities, and classifications find their roots in the pioneering works of Whitney, Thom, Boardmann, and  Mather~\cite{whitney1955singularities,boardman1967singularities,Tho69,
mather1973stratifications,
mather1973thom} on smooth maps between Euclidean spaces. One such problem, initiated by Thom~\cite{thom1962stabilite,Tho69}, asks for a topological characterization: Two continuous maps $f,g:\CtCmn$ are said to be \emph{topologically equivalent} if there exist homeomorphisms $\varphi:\C^m\to\C^m$ and $\psi:\CtCn$ satisfying 
\[
\psi\circ f\circ \varphi = g.
\] Fukuda~\cite{Fuk76} confirmed Thom's famous conjecture that there are finitely many topological types of polynomial functions $\C^m\to\C$ with a given degree $d$. Soon after, Aoki, Noguchi, and Sabbah~\cite{aoki1980topological,Sab83} generalized this result for polynomial maps $\C^2\to\C^n$ with fixed $d=\deg f:=\max(\deg f_1,\ldots,\deg f_n)$. Nakai then showed~\cite{Nak84} that all the remaining cases for $m\geq 3$ admit infinitely many topological types if $\deg f\geq 4$. 

The ensuing classification problem of topological types is still open for polynomial functions with a given degree, as well as for maps on the plane. Nevertheless, it becomes tractable in some instances, such as proper, or quadratic maps. In the former case, this amounts to counting degree-$d$ ramified covers of $\C^m$ (c.f.~\cite[Proof of Theorem 3.5]{jelonek2016semi}), and in the latter, a full classification has been achieved through case-by-case analysis~\cite{FJ17,FJM18}. In general, however, the difficulty emanates from the lack of accessible tools for describing those properties of polynomial maps that are invariant under homeomorphisms. 

Our main result concerns polynomial maps on the plane.
\begin{theorem}\label{thm:main-bound}
    The number of topological types of polynomial maps $\CtC$ of degree smaller or equal to $3$ and $4$ are at least $26$ and
    $3~217$ respectively.
    %$11~604$ respectively.
\end{theorem} To achieve the lower bounds in Theorem~\ref{thm:main-bound}, we developed a tool for constructing many topological types that relies on polyhedral invariants of polynomials; the \emph{support} of a polynomial $P$ in $n$ variables is the set $\supp(P)\subset\N^n$ of exponent vectors of the monomials of $P$. Its \emph{Newton polytope} $\NP(P)$ is the convex hull in $\R^n$ of $\supp(P)$. 
Let $\ccPgeq$ be the set of all pairs $A:=(A_1,A_2)$ of lattice polytopes in the non-negative orthant of $\R^2$. For any $A\in\ccPgeq$, we use $\C^A$ to denote the vector space consisting of polynomial maps $f:\CtC$ that satisfy $\NP(f_1)\subset A_1$ and $\NP(f_2)\subset A_2$. 

As a consequence of Sabbah's results in~\cite{Sab83} (see Lemma~\ref{lem:main1} below), there exists a Zariski open subset $\Omega\subset \C^A$ in which any two polynomial maps are topologically equivalent. Accordingly, a map $f\in\C^A$ is called \emph{generic} if $f\in\Omega$. Unsurprisingly, all generic maps in $\C^A$ share a list of invariants, which can be computed using computer algebra methods (see Corollary~\ref{cor:main} below). This list consists of algebraic identities from three curves: the \emph{critical points}, $\cCf$, of $f$, its \emph{discriminant} $\cDf:=f(\cCf)$, and its \emph{non-properness} set $\cS_f$, which is the set
\begin{align*}
\cS_f&:=\left\{w\in\C^2~\left|~\exists~\{z^{k}\}_{k\in\N}\subset\C^2,~ \Vert z^k\Vert\underset{k\rightarrow\infty}{\longrightarrow}\infty,~ f(z^k)\longrightarrow w\right\}\right.
\end{align*} Whitney studied in~\cite{whitney1955singularities} smooth maps of the plane into the plane, and showed that $\cCf$ is smooth, and a singularity of $\cDf$ is either a \emph{simple node} or an \emph{ordinary cusp}. In other words, either $p$ is a normal self-crossing of $\cDf$, or, for some $a\in\C^2$ with $f(a)=p$, the germ $(f,a):(\C^2,a)\to (\C^2,f(a))$ is biholomorphically equivalent to $(\C^2,0)\ni(x,~y)\mapsto (x,~xy+y^3)\in(\C^2,0)$. 

A shift from analytic to polynomial maps arises as a highly challenging task due to the latter forming a more rigid class. Nevertheless, Farnik, Jelonek and Ruas~\cite{FJR19} recently extended Whitney's results for the space $\Omega(d_1,d_2)$ of polynomial maps on the plane $f:\CtC$ satisfying $\deg f_1 = d_1$ and $\deg f_2 = d_2$. Namely, they showed that for any given pair $(d_1,d_2)\in\N^2$, a generic map in $ \Omega(d_1,d_2)$ is proper (i.e., $\cS_f=\emptyset$), the curve $\cCf$ is smooth, a singular point of $\cDf$ can either be a simple node or an ordinary cusp, and their respective numbers are computed in terms of $d_1$ and $d_2$. 

In this paper, we extend here some of the above results in~\cite{FJR19} to maps in $\C^A$ for all $A$ in a large family $\ccCgeq\subset \ccPgeq$ of pairs of polytopes. The family $\ccCgeq$ strictly contains all $(A_1,A_2)$ where, for each $i=1,2$, there is a non-singular plane conic passing through five points in $A_i\cap \N^2\setminus\{(0,0)\}$. Accordingly, all pairs in $\ccCgeq$ will be called ``conical'' (Definition~\ref{def:conical}). Our central result, Theorem~\ref{thm:discriminant}, on generic maps $f\in\C^A$ states that $\cCf$ is smooth outside $\{(0,0)\}$, and if $p\in\C^2\setminus\{(0,0)\}$ is a singular point of $\cDf$, then it is either a simple node or an ordinary cusp. We also show that a singular point of $\cS_f\setminus\{(0,0)\}$ is always a simple node. Furthermore, we express the number of the above singularities in terms of the combinatorics of $A$, and we present formulas for other topologically relevant data of $f$ such as the number of components of $\cCf$, $\cDf$, and $\cS_f$, together with their multiplicities and geometric genera.  Consequently, we obtain a list of numerical invariants that distinguish the topological type of $f$, and represent it as a vector $\Psi(A)\in\N^{12}$, which we call the ``polyhedral type'' of $f$ (see Definition~\ref{def:polyhrdral-type}), given by a map $\Psi:\ccCgeq\to\N^{12}$. To this end, we show (Theorem~\ref{thm:main2}) that for any $A,B\in\ccCgeq$, two generic maps $f\in\C^A$ and $g\in\C^B$ are not topologically equivalent whenever $f$ and $g$ have different polyhedral types. 

\begin{corollary}\label{cor:lower_bound}
Let $d\in\N$ be any positive integer, and let $\mathscr{C}_d$ be the subset of conical pairs $A\in\ccCgeq$ for which any $f\in\C^A$ satisfies  $\deg f = d$. Then, the value $\#\Psi(\mathscr{C}_d)$ bounds from below the number of topological types of polynomial maps $\CtC$ of degree $d$.
\end{corollary}

Theorem~\ref{thm:main-bound} is obtained by counting images $\Psi(A)$ over the all conical pairs $A$ from $\mathscr{C}_3$ and $\mathscr{C}_4$.

\begin{remark}\label{rem:asympt}
The sharpness of the lower bound in Theorem~\ref{thm:main-bound} is expected to deteriorate significantly for large degrees. Indeed, on the one hand, thanks to the properties of the well-known Lyashko-Looijenga map~\cite{looijenga1974complement,lyashko1983geometry}, there are at least $d^{d-3}$ topological types of polynomial functions $g:\C\to\C$ of degree $d$~\cite[\S 5.1.3]{lando2004graphs}. By considering non-dominant maps of the form $(x,~y)\mapsto (0,~g(x))$, we deduce that the lower bound on the number of topological types of polynomial maps $\CtC$ of degree $d$ is at least exponential in $d$. On the other hand, the lower bound in Theorem~\ref{thm:main-bound} is lower than the number of pairs $A\subset \mathscr{C}_d$, which is at most polynomial in $d$ \cite{BOUSQUETMELOU19961}.
\end{remark}

\subsection*{Idea of proof and organization of the paper}\label{sub:proof_idea} 
\S\ref{sec:char-map} is devoted to elaborating on topological invariants (Definition~\ref{def:topological-invariant}), polyhedral type of a generic map (Definition~\ref{def:polyhrdral-type}), and detailing the main result, Theorem~\ref{thm:discriminant}, for conical pairs $A\in\ccCgeq$. 
The proof of Theorem~\ref{thm:discriminant} is achieved in~\S\ref{sec:proof_of_theorem}; it requires three key results: Theorems~\ref{thm:critical-smooth}, \ref{thm:discr-infty}, and~\ref{thm:non-properness-long} in~\S\ref{sec:affine-sing},~\S\ref{sec:toric_sing}, and~\S\ref{sec:non-prop} respectively. 

The main result of~\S\ref{sec:affine-sing} is Theorem~\ref{thm:critical-smooth}; it states that the critical locus $\cCf$ of a generic map $f\in\C^A$ is smooth outside $\{(0,0)\}$, and that $\cDf$ has at most mild singularities. For its proof, we show that the restricted Thom-Boardman map sending $\C^A$ to the first few jet-spaces is smooth with respect to the Thom-Boardman strata describing the singularities of $f$. It turns out that the conical property of $A$ is necessary for smoothness. Then, similarly as in~\cite[Sections 2, 3 and 4]{FJR19}, we apply classical results in the theory of jet spaces (see e.g.,~\cite{golubitsky2012stable}) to describe the singularities of $\cDf$ in $\C^2\setminus\{(0,0)\}$. As for counting their number, one necessitates an adjunction-type formula relating the singularities of $\cDf$ to its geometric genus, together with singular points ``at infinity''. The latter are sent to the boundary of the compactification of $\cDf$ in the toric variety $X_\Delta$ corresponding to the Newton polytope $\Delta:=\NP(\cDf)$. Theorem~\ref{thm:discr-infty} in~\S\ref{sec:toric_sing} states that there are no boundary singular points under this compactification. To prove this, we use the fact that if $\Sigma$ is the Newton polytope of $\cCf$, then its toric compactification in $X_\Sigma$ is transversal to all the orbits of $X_\Sigma$. 

The last ingredient for proving Theorem~\ref{thm:discriminant} requires describing the non-properness set $\cS_f$; it was shown in~\cite{EHT21} that for a generic $f\in\C^A$, the polynomial defining $\cS_f$ depends only on those coefficients of $f$ whose exponent vectors appear at pairs of faces of $A$ (see Theorem~\ref{thm:non-prop-main} below). In Theorem~\ref{thm:non-properness-long}, we use this description to show that $\cS_f$ has only simple nodes as singularities in $\C^2\setminus\{(0,0)\}$. We also show that if $\Xi$ is the Newton polytope of $\cS_f$, then its toric compactification in $X_\Xi$ is transversal to all the one-dimensional orbits of $X_\Xi$. 

Altogether the above descriptions of $\cCf$, $\cDf$, and $\cS_f$ imply that one can use classical results of Bernstein, Kouchnirenko, and Khovanskii, presented in~\S\ref{sec:proof_of_theorem}, to express the singularities and the geometric genera in terms of lattice points and volumes obtained from $\Sigma$, $\Delta$, and $\Xi$. This forms the proof of Theorem~\ref{thm:discriminant}. 

Polytopes $\Sigma,\Delta,\Xi\subset\R^2$ are determined solely from the pair $A$ for generic $f\in\C^A$. The methods for computing them, together with the corresponding algorithmic implementations are presented in~\S\ref{sec:algorithms}. Our software and data 
are made available in the {\tt MathRepo} collection at MPI-MiS via
 \url{https://mathrepo.mis.mpg.de/PolyhedralTypesOfPlanarMaps}.

\section{Topological invariants from Newton pairs}\label{sec:char-map}  
This section is divided into two parts. In the first part, we prove Corollary~\ref{cor:main}, and in the second, we show Theorem~\ref{thm:main2}.

\subsection{The polyhedral type of generic maps} We start with the following result. 

\begin{lemma}\label{lem:main1}
For any pair of polytopes $A\in\ccPgeq$, there exists a Zariski open subset $\Omega\subset \C^A$ in which any two elements $f,g\in\Omega$ represent two topologically equivalent polynomial maps $\CtC$.
\end{lemma}

\begin{proof}
Consider the space $\C^{N_d}$ of polynomials maps $f:\CtC$ having degree at most $d$. On the one hand, for every $A\in\ccCgeq$, the space $\C^A$ has a canonical embedding $\C^A\hookrightarrow\C^{N_d}$, as the set of all maps $g$ with some coefficients are substituted by zero. We identify $\C^A$ with a coordinate linear subspace of $\C^{N_d}$. On the other hand it is known (see e.g.,~\cite[pp. 5]{Sab83}) that for every affine subset $S\subset\C^{N_d}$, there exists a Zariski open $\Omega\subset S$ in which any two polynomial maps $f,g\in\Omega $ in a connected component of $\Omega$ will be topologically equivalent. To conclude the proof, notice that every Zariski open of $\C^A$ is connected. 
\end{proof}
\begin{definition}\label{def:generic}
In the notation of Lemma~\ref{lem:main1}, a map $f\in\C^A$ is called \emph{generic} if it belongs to $\Omega$.
\end{definition}

Next, we describe the numerical data that are shared among topologically equivalent maps. Let us first introduce some notations and definitions.
\begin{definition}\label{def:multipl-singularity}
Let $C\subset\C^2$ be a planar algebraic curve given by a bivariate polynomial $P\in\C[z_1,z_2]$. The smallest $m\in\N$ for which the homogeneous part, $P_m$, of $P$ is non-zero is called the \emph{multiplicity} of $P$ at the origin. We will denote it by $m_0(P)$ or $m_0(C)$. Any linear factor of $P_m$ is called a \emph{tangent} to $C$ at the origin. We define the \emph{multiplicity} $m_p(P)$ of $P$ at $p$ as $m_0(\tilde{C})$, with $\tilde{C}$ being the translation of $C$ that sends $p$ to the origin. We say that $p$ is a \emph{singular point} of $C$ if $P = \partial P/\partial z_1 = \partial P/\partial z_2 = 0$, and smooth otherwise. We have $m_p(P) = 1$ if and only if $p$ is smooth.
\end{definition}
A \emph{simple node}, for example, is an isolated singularity satisfying $m_p(P) = 2$ and such that $C$ has exactly two different tangents at $p$.

\begin{definition}\label{def:int-mult-Milnor}
Let $C,D\subset\C^2$ be two curves defined by polynomials $P,Q\in\C[z_1,z_2]$. For a point $p \in C\cap D$, we define the \emph{intersection multiplicity} of $P$ and $Q$ at $p$ to be
\begin{align*}
\mu_p(P,Q)&:=\dim_\C\mathscr{O}_p/\langle P,~Q\rangle \in\N\cup\{\infty\},
\end{align*} where $\mathscr{O}_p$ is the local ring at $p$. For a point $p \in C$, we define the \emph{Milnor number} of $C$ at $p$ to be
\begin{align*}
\mu_p(P)&:=\dim_\C\mathscr{O}_p/\langle \partial P/\partial z_1,~ \partial P/\partial z_2\rangle \in\N\cup\{\infty\}.
\end{align*}
\end{definition} 

If, for example, $p\in C$ is a singularity satisfying $m_p(P) = 2$, with only one tangent $L$ to $C$ at $p$, and $\mu(P,L)=3$, then $p$ is called an \emph{ordinary cusp}. A \emph{mild singularity} of an algebraic curve $C\subset\C^2$ is either an ordinary cusp or a simple node. 

\begin{remark}\label{rem:cusp-node_Milnor}
It is well known that $\mu_p(P)=1$ if $p$ is a simple node, and $\mu_p(P)=2$ if it is an ordinary cusp. 
\end{remark} 

Let $f:\C^2\ni z:=(z_1,z_2)\to (f_1(z)~,f_2(z))\in\C^2$ be a polynomial map. Each of the critical locus $\cCf:=\{\det\Jac_z f=0\}$, discriminant $\cDf$, and non-properness set $\cS_f$ is a Zariski closed in $\C^2$, with $\cCf$ and $\cS_f$ always being curves~\cite{Jel93}. Then, there exists an integer $k\in\N$, satisfying $\#f^{-1}(w) = k$ for any $w\in\C^2\setminus\cDf\cup\cS_f$. We call $k$ the \emph{topological degree} of $f$.

\begin{definition}\label{def:top-mult-crit}
Given an irreducible component $C$ of $\cCf$
, the map $f_{|_U}:U\to f(U)$ is a ramified cover of degree $m$ for any small enough neighborhood $U$ of any point $p\in C$. The \emph{topological multiplicity} of $C$ is defined to be the value $m-1$. The \emph{topological multiplicity} of an irreducible component $K\subset\cDf$ is defined to be the sum of all topological multiplicities of components $C\subset \cCf$ mapping to $K$. We use the convention that the topological multiplicity of an empty set is always zero.

For any irreducible component $S\subset\cS_f$ there is a finite set $\sigma\subset S$ in which for any two $w,\tilde{w}\in S\setminus \sigma$, we have $\# f^{-1}(w)=\# f^{-1}(\tilde{w})$~\cite{EHT21}. Any $w\in S\setminus\sigma$ is called  \emph{generic point of $S$}. The \emph{topological multiplicity} of an irreducible component of $\cS_f$ is the difference $k-\#f^{-1}(w)$, where $w$ is any generic point of $S$ and $k$ is the topological degree of $f$.
\end{definition}

Let $\Irr(K)$ denote the set of irreducible components of an algebraic curve $K$, and let $\Sing(K)$ denote the set of its isolated singularities. The following result is proven at the end of this section.

\begin{proposition}\label{prp:separator}
Let $f,g:\CtC$ be two polynomial maps satisfying
\begin{equation}\label{dia:commutative}
\psi\circ f\circ\varphi = g,
\end{equation} for some homeomorphisms $\varphi,\psi:\C^2\to\C^2$. Then, the following assertions hold:

\begin{enumerate}[label=(\alph*)]

	\item\label{it:top-degree} The topological degree of $f$ is equal to the topological degree of $g$.\\
	
	\item\label{it:Jelonek-components} The maps below are bijections 
	that preserve the topological multiplicities and geometric genera:
\begin{align*}
\Irr(\cCg)& \longrightarrow\Irr(\cCf),&& &
\Irr(\cDf)& \longrightarrow\Irr(\cDg), && &
\Irr(\cS_f)&\longrightarrow\Irr(\cSg),\\ 
K&\longmapsto \varphi(K), && &
K&\longmapsto \psi(K), && &
K&\longmapsto \psi(K).
\end{align*}

		\item\label{it:jacobian-multiplicity} For any subsets $ \kappa\subset\Irr(\cCg)$, $ \delta\subset\Irr(\cDf)$, and $ \sigma\subset\Irr(\cS_f)$, forming unions of irreducible algebraic curves in $\C^2$, the maps below are bijections that preserve the Milnor numbers:
\begin{align*}
\Sing(\kappa)& \longrightarrow\Sing(\varphi (\kappa)),&& & 
\Sing(\delta)&\longrightarrow\Sing(\psi(\delta)), && &
\Sing(\sigma)&\longrightarrow\Sing(\psi(\sigma)),\\ 
p&\longmapsto \varphi(p), && &
p&\longmapsto \psi(p), && &
p&\longmapsto \psi(p).
\end{align*} 	
\end{enumerate} 
\end{proposition}

\begin{lemma}\label{lem:homeomorph=singul}
Let $C,D\subset \C^2$ be two algebraic curves in $\C^2$, and let $\varphi:\CtC$ be a homeomorphism sending $C$ to $D$. Then, there exists a bijection $\Irr(C)\longrightarrow\Irr(D)$, $K\longmapsto \varphi(K)$, that preserves the geometric genera, and a bijection $\Sing(C)\longrightarrow\Sing(D)$, $p\longmapsto \varphi(p)$ that preserves the Milnor numbers. 
\end{lemma}

\begin{proof}
 For any small ball $U$ containing $p$, the pair $(\partial U,~\partial U \cap C)$ is homeomorphic to $(\partial \varphi(U),~\partial \varphi(U) \cap D)$. Then, the singular points $p$ and $\varphi(p)$ of $C$ and $D$ respectively are \emph{link equivalent} in the sense of~\cite{Saeki89}. Then~\cite[Corollary 2]{Saeki89} implies the existence of a Milnor-number  preserving bijection $\Sing(C)\to\Sing(D)$. 

Consider the two resolutions of singularities $\bl_1:X\to\C^2$, $\bl_2:Y\to\C^2$ obtained as blowups of $\C^2$ at $\sigma:=\Sing(S)$ and at $\tau:=\Sing(D)$ respectively. Let $\tilde{C}\subset X$, $\tilde{D}\subset Y$ be the corresponding strict transforms of $C,D\subset\C^2$, and $\Sigma_1,\Sigma_2$ be the collections of the exceptional divisors $\bl^{-1}_1(\sigma)$, $\bl^{-1}_2(\tau)$ respectively. Then, the restrictions $\bl_{1|_{X\setminus \Sigma_1}}$ and $\bl_{2|_{Y\setminus \Sigma_2}}$ are homeomorphisms sending $\tilde{C}$ to $C$ and $\tilde{D}$ to $D$. So is the restricted map $\phi:=\varphi_{|_{\C^2\setminus \sigma}}:\C^2\setminus\sigma\to\C^2\setminus\tau$ a homeomorphism. Therefore, the map 
\[
F:=\bl^{-1}_{2|_{Y\setminus \Sigma_2}}~\circ~\phi~\circ~\bl_1:~~X\setminus \Sigma_1\longrightarrow Y\setminus \Sigma_2
\] is also a homeomorphism sending $\tilde{C}\setminus \Sigma_1$ onto $\tilde{D}\setminus \Sigma_2$. Since irreducible components of $\tilde{C}$ and $\tilde{D}$ are smooth and disjoint, the map $F$ extends to a homeomorphism $X\to Y$ that sends irreducible components of $\tilde{C}$ onto those of $\tilde{D}$. This shows that $\varphi_{|_{C}}$ decomposes onto several homeomorphisms $\varphi_1,\ldots,\varphi_k$, one for each irreducible component $C_i\subset C$, where $\varphi_i(C_i)$ is an irreducible component of $D$. This gives rise to a bijection $\Irr(C)\to\Irr(D)$ that preserves the geometric genera. 
\end{proof}

\begin{proof}[Proof of Proposition~\ref{prp:separator}]
Item~\ref{it:top-degree} is obvious. As for the other items, we start with $\cS_f$.

From the definitions, for any $q\in\cS_f $ and neighborhood $V\ni q$, the image of the Euclidean closure $f^{-1}(\overline{V})$ is also not compact. Thus, so is the set $g^{-1}(\psi(\overline{V}))$ as it coincides with $\varphi^{-1}(f^{-1}(\overline{V}))$. This shows that $\psi(q)\in \cS_g$. Since $\psi^{-1}$ is also a homeomorphism, we obtain that $\cS_g = \psi(\cS_f)$. Then, we get Item~\ref{it:jacobian-multiplicity} for $\cS_f$ by applying  Lemma~\ref{lem:homeomorph=singul} iteratively on each subset $\sigma\subset\Irr(\cS_f)$. To show Item~\ref{it:Jelonek-components}, note that for every $K\in\Irr(\cS_f)$, we get $\psi(K)\in\Irr(\cS_g)$, and thus $\#f^{-1}(w)=\#g^{-1}(\psi(w))$ holds for a generic point $w\in K$. This concludes that $K$ and $\psi(K)$ have the same topological multiplicities.

Now we show Items~\ref{it:Jelonek-components} and~\ref{it:jacobian-multiplicity} only for the critical points as the proof for the discriminants follows similar steps. Let $K\in\Irr(\cCf)$ and let $m$ be its topological multiplicity. Then, for any $p\in C$ outside any other component of $\cCf$, and for any small-enough neighborhood $U$ of $p$, the map 
\begin{equation}\label{eq:equiv-covers}
(\psi\circ f\circ \varphi)_{|_{\tilde{U}}}=g_{|_{\tilde{U}}}
\end{equation} is an unramified cover of $\tilde{V}:=g(\tilde{U})$ of degree $m$, where $\tilde{U}:=U\setminus K$. Since $\varphi$ and $\psi$ are homeomorphisms, the map $f_{|_{\varphi(\tilde{U})}}$ is also an unramified cover of $\psi(\tilde{U})$ of degree $m$. Therefore, we get $\varphi(p)\in \cCf$. Similarly, as before, we deduce that $\cCf= \varphi(\cC_g)$. Then, Items~\ref{it:Jelonek-components} and~\ref{it:jacobian-multiplicity} follow thanks to Lemma~\ref{lem:homeomorph=singul}. 
\end{proof} Let $\cMCt$ denote the space of all polynomial maps $\CtC$. The following result is a straightforward consequence of Proposition~\ref{prp:separator}. 
\begin{corollary}\label{cor:main}
There exists a (possibly infinite) index set $I\subset\R$ and a map
\[
\Upsilon:\cMCt\longrightarrow\N^I,
\] sending $f\in\cMCt$ onto a list of values encoding all its numerical invariants from Proposition~\ref{prp:separator}, such that if $\Upsilon(f)\neq\Upsilon(g)$, then $f$ and $g$ are not topologically equivalent.
\end{corollary}
\begin{example}\label{ex:map-upsilon}
One may define the image $\Upsilon:=(\Upsilon_i)_{i\in I}$ so that each $i\in I$ labels a pair $(O,P)$ of one \emph{objects} $O$ corresponding to $f$ that satisfy a certain \emph{property} $P$, while the coordinate $\Upsilon_i(f)$ computes the number of the above pairs. An object $O$ can be either a singularity, or a tuple of irreducible components of curves appearing in Proposition~\ref{prp:separator}, and $P$ is a numerical value that for the Milnor number, geometric genus, topological multiplicity, or a given number of intersection points with a given multiplicity. For instance, for one of the indexes $i_0\in I$, the value $\Upsilon_{i_0}(f)$ is defined as the number of sets $\{X,Y\}$ of two irreducible components of $\cS_f$ (this is $O_{i_0}$) for which the intersection $X\cap Y$ consists of exactly three isolated points of multiplicity five (this is $P_{i_0}$). 
\end{example}
As we will show in Theorem~\ref{thm:discriminant} below, generic polynomial maps $f\in\cMCt$ have a tamed topology. This makes all but finitely many coordinates of $\Upsilon(f)$ in Corollary~\ref{cor:main} superfluous. 

\begin{definition}\label{def:topological-invariant}
A \emph{topological invariant} of a map $f\in\cMCt$ is the image $F\circ\Upsilon(f)\in\N^n$, where $F:\N^I\longrightarrow\N^n$ is any map.
\end{definition} 

\subsection{The polyhedral type for conical pairs}\label{sub:polyhedral_type}
We start by introducing several combinatorial identities.

\subsubsection{Newton number of singularities}\label{sss:not-prelim}
A \emph{polytope}, $\Pi$ in $\R^n$ is the convex hull of finitely many points $S\subset\R^n$. We say that $\Pi$ is a  \emph{lattice polytope}, if its vertices are points in $\Z^n$. 
A hyperplane $H$ is said to be \emph{supporting $\Pi$} if the closure of one of its half-spaces contains $\Pi$. This hyperplane determines a \emph{face} $F$ of $\Pi$, denoted $F\prec \Pi$, as the result of the intersection $\Pi\cap H$. Hence, $F$  minimizes a function $\alpha^*:\Pi\longrightarrow\R$, given by $(x_1,\ldots,x_n)\longmapsto \alpha_1x_1+\cdots+\alpha_nx_n$, where $\alpha:=(\alpha_1,\ldots,\alpha_n)$ is the normal vector of $H$ directed towards $\Pi$. We thus say that $\alpha$ \emph{supports} $F$, and we use $\Pi^\alpha$ in reference to $F$.

 We set $\Vol (\cdot)$ to be the scaled Euclidean volume of any object in $\R^2$, so that $\Vol(\sigma)=1/2$ if $\sigma$ is the simplex with vertices $(0,0)$, $(0,1)$, and $(1,0)$. Let $\Pi$ be the Newton polytope of a polynomial $P\in\C[z_1,z_2]$ and assume in what follows that $P$ is \emph{convenient}. That is, $\Pi$ intersects both coordinate axes of $\R^2$. Then $\Rtg\setminus\Pi $ is a union of two connected components, one of which, denoted by $\Sigma_0$, is bounded.  Let $a$ and $b$ be the distance from $\bm{0}$ to the common part of $\Pi$ intersected with the first and second coordinate axis respectively (see e.g., Figure~\ref{fig:N012}). We define the \emph{Newton number} $\cN(\Pi)$, or alternatively, $\cN(P)$, of $\Pi$ to be zero if $\Sigma_0$ is empty and otherwise we set
\[
\cN(\Pi):=2\cdot \Vol(\Sigma_0) -a-b+1.
\] Recall that $\mu_0(P)$ denotes the Milnor number of $P$ at the origin $\bm{0}\in\C^2$ (see Definition~\ref{def:multipl-singularity}). A convenient polynomial with Newton number zero, for example, can only be in one of the two rightmost cases represented in Figure~\ref{fig:N012}. That is, its germ at $\bm{0}$ is biholomorphic to $(x,y)\to(x^n+y)$. This gives the following observation.

\begin{lemma}\label{lem:Milnor=0}
A convenient polynomial $P\in\C[z_1,~z_2]$ satisfies $\cN(P)=0$ $\Longleftrightarrow$ $\mu_0(P)=0$.
\end{lemma}

\begin{figure}[htb]

\tikzset{every picture/.style={line width=0.75pt}} %set default line width to 0.75pt        

\begin{tikzpicture}[x=0.75pt,y=0.75pt,yscale=-1,xscale=1]
%uncomment if require: \path (0,292); %set diagram left start at 0, and has height of 292

%Shape: Polygon [id:ds48615669733805433] 
\draw  [color={rgb, 255:red, 0; green, 0; blue, 0 }  ,draw opacity=0 ][fill={rgb, 255:red, 65; green, 117; blue, 5 }  ,fill opacity=0.4 ][line width=0.75]  (277,210.84) -- (217,211.03) -- (205,103) -- (277,103.02) -- cycle ;
%Shape: Polygon [id:ds40714620993927386] 
\draw  [color={rgb, 255:red, 0; green, 0; blue, 0 }  ,draw opacity=0 ][fill={rgb, 255:red, 65; green, 117; blue, 5 }  ,fill opacity=0.4 ][line width=0.75]  (305.86,138.73) -- (305.67,198.73) -- (413.7,210.73) -- (413.7,138.73) -- cycle ;
%Straight Lines [id:da47164198541024926] 
\draw [color={rgb, 255:red, 155; green, 155; blue, 155 }  ,draw opacity=1 ]   (205,210.84) -- (276.57,210.84) ;
%Straight Lines [id:da6938958585763035] 
\draw [color={rgb, 255:red, 155; green, 155; blue, 155 }  ,draw opacity=1 ][fill={rgb, 255:red, 155; green, 155; blue, 155 }  ,fill opacity=0.3 ]   (205,210.84) -- (205,103.35) ;
%Shape: Circle [id:dp4443756873440333] 
\draw  [fill={rgb, 255:red, 0; green, 0; blue, 0 }  ,fill opacity=1 ] (215.33,211.03) .. controls (215.33,210.11) and (216.08,209.36) .. (217,209.36) .. controls (217.92,209.36) and (218.67,210.11) .. (218.67,211.03) .. controls (218.67,211.95) and (217.92,212.7) .. (217,212.7) .. controls (216.08,212.7) and (215.33,211.95) .. (215.33,211.03) -- cycle ;
%Shape: Circle [id:dp8750551404054585] 
\draw  [fill={rgb, 255:red, 0; green, 0; blue, 0 }  ,fill opacity=1 ] (203.33,103) .. controls (203.33,102.08) and (204.08,101.33) .. (205,101.33) .. controls (205.92,101.33) and (206.67,102.08) .. (206.67,103) .. controls (206.67,103.92) and (205.92,104.67) .. (205,104.67) .. controls (204.08,104.67) and (203.33,103.92) .. (203.33,103) -- cycle ;
%Straight Lines [id:da9921852094316249] 
\draw    (217,211.03) -- (277,210.84) ;
%Straight Lines [id:da5550263452584204] 
\draw    (217,211.03) -- (205,103) ;
%Straight Lines [id:da9102079345675397] 
\draw [color={rgb, 255:red, 155; green, 155; blue, 155 }  ,draw opacity=1 ]   (305.86,210.73) -- (305.86,139.16) ;
%Straight Lines [id:da08263945643841375] 
\draw [color={rgb, 255:red, 155; green, 155; blue, 155 }  ,draw opacity=1 ][fill={rgb, 255:red, 155; green, 155; blue, 155 }  ,fill opacity=0.3 ]   (305.86,210.73) -- (413.35,210.73) ;
%Shape: Ellipse [id:dp3173785274672587] 
\draw  [fill={rgb, 255:red, 0; green, 0; blue, 0 }  ,fill opacity=1 ] (305.67,200.4) .. controls (306.59,200.4) and (307.34,199.65) .. (307.34,198.73) .. controls (307.34,197.81) and (306.59,197.07) .. (305.67,197.07) .. controls (304.75,197.07) and (304,197.81) .. (304,198.73) .. controls (304,199.65) and (304.75,200.4) .. (305.67,200.4) -- cycle ;
%Shape: Ellipse [id:dp9330031887330232] 
\draw  [fill={rgb, 255:red, 0; green, 0; blue, 0 }  ,fill opacity=1 ] (413.7,212.4) .. controls (414.62,212.4) and (415.37,211.65) .. (415.37,210.73) .. controls (415.37,209.81) and (414.62,209.07) .. (413.7,209.07) .. controls (412.78,209.07) and (412.03,209.81) .. (412.03,210.73) .. controls (412.03,211.65) and (412.78,212.4) .. (413.7,212.4) -- cycle ;
%Straight Lines [id:da4226494776287181] 
\draw    (305.67,198.73) -- (305.86,138.73) ;
%Straight Lines [id:da5035585441529483] 
\draw    (305.67,198.73) -- (413.7,210.73) ;
%Straight Lines [id:da2571709008030316] 
\draw [color={rgb, 255:red, 155; green, 155; blue, 155 }  ,draw opacity=1 ]   (60,211) -- (155.67,211) ;
%Straight Lines [id:da26073771002367263] 
\draw [color={rgb, 255:red, 155; green, 155; blue, 155 }  ,draw opacity=1 ]   (60,211.21) -- (60,127.09) ;
%Shape: Circle [id:dp4483551677677662] 
\draw  [fill={rgb, 255:red, 0; green, 0; blue, 0 }  ,fill opacity=1 ] (58.33,127) .. controls (58.33,126.08) and (59.08,125.33) .. (60,125.33) .. controls (60.92,125.33) and (61.67,126.08) .. (61.67,127) .. controls (61.67,127.92) and (60.92,128.67) .. (60,128.67) .. controls (59.08,128.67) and (58.33,127.92) .. (58.33,127) -- cycle ;
%Shape: Circle [id:dp9475378264375677] 
\draw  [fill={rgb, 255:red, 0; green, 0; blue, 0 }  ,fill opacity=1 ] (154.33,210.4) .. controls (154.33,209.48) and (155.08,208.73) .. (156,208.73) .. controls (156.92,208.73) and (157.67,209.48) .. (157.67,210.4) .. controls (157.67,211.32) and (156.92,212.07) .. (156,212.07) .. controls (155.08,212.07) and (154.33,211.32) .. (154.33,210.4) -- cycle ;
%Shape: Circle [id:dp6300753568333781] 
\draw  [fill={rgb, 255:red, 0; green, 0; blue, 0 }  ,fill opacity=1 ] (106.11,162.92) .. controls (106.11,162) and (106.86,161.26) .. (107.78,161.26) .. controls (108.7,161.26) and (109.44,162) .. (109.44,162.92) .. controls (109.44,163.84) and (108.7,164.59) .. (107.78,164.59) .. controls (106.86,164.59) and (106.11,163.84) .. (106.11,162.92) -- cycle ;
%Shape: Polygon [id:ds568544374224443] 
\draw  [fill={rgb, 255:red, 245; green, 166; blue, 35 }  ,fill opacity=0.4 ][line width=0.75]  (107.78,162.92) -- (156,210.4) -- (108.11,199.26) -- (84.11,186.59) -- (72.11,162.26) -- (60,127) -- cycle ;
%Shape: Circle [id:dp03920881570815837] 
\draw  [fill={rgb, 255:red, 0; green, 0; blue, 0 }  ,fill opacity=1 ] (82.44,186.59) .. controls (82.44,185.67) and (83.19,184.92) .. (84.11,184.92) .. controls (85.03,184.92) and (85.78,185.67) .. (85.78,186.59) .. controls (85.78,187.51) and (85.03,188.26) .. (84.11,188.26) .. controls (83.19,188.26) and (82.44,187.51) .. (82.44,186.59) -- cycle ;
%Shape: Circle [id:dp8475357877938046] 
\draw  [fill={rgb, 255:red, 0; green, 0; blue, 0 }  ,fill opacity=1 ] (106.44,199.26) .. controls (106.44,198.34) and (107.19,197.59) .. (108.11,197.59) .. controls (109.03,197.59) and (109.78,198.34) .. (109.78,199.26) .. controls (109.78,200.18) and (109.03,200.92) .. (108.11,200.92) .. controls (107.19,200.92) and (106.44,200.18) .. (106.44,199.26) -- cycle ;
%Shape: Circle [id:dp1982311959623847] 
\draw  [fill={rgb, 255:red, 0; green, 0; blue, 0 }  ,fill opacity=1 ] (70.44,162.26) .. controls (70.44,161.34) and (71.19,160.59) .. (72.11,160.59) .. controls (73.03,160.59) and (73.78,161.34) .. (73.78,162.26) .. controls (73.78,163.18) and (73.03,163.92) .. (72.11,163.92) .. controls (71.19,163.92) and (70.44,163.18) .. (70.44,162.26) -- cycle ;

% Text Node
\draw (158,213.8) node [anchor=north west][inner sep=0.75pt]  [font=\tiny]  {$a$};
% Text Node
\draw (60,214.4) node [anchor=north] [inner sep=0.75pt]  [font=\tiny]  {$0$};
% Text Node
\draw (58,127) node [anchor=east] [inner sep=0.75pt]  [font=\tiny]  {$b$};
% Text Node
\draw (67.5,194.2) node [anchor=north west][inner sep=0.75pt]  [font=\tiny]  {$\Sigma _{0}$};
% Text Node
\draw (205,214.4) node [anchor=north] [inner sep=0.75pt]  [font=\tiny]  {$0$};
% Text Node
\draw (201.33,103) node [anchor=east] [inner sep=0.75pt]  [font=\tiny]  {$b$};
% Text Node
\draw (413.7,215.8) node [anchor=north] [inner sep=0.75pt]  [font=\tiny]  {$a$};
% Text Node
\draw (305.86,214.13) node [anchor=north] [inner sep=0.75pt]  [font=\tiny]  {$0$};
% Text Node
\draw (302,198.73) node [anchor=east] [inner sep=0.75pt]  [font=\tiny]  {$1$};
% Text Node
\draw (219,214.43) node [anchor=north west][inner sep=0.75pt]  [font=\tiny]  {$1$};

\end{tikzpicture}

\caption{Some examples of convenient polytopes.}\label{fig:N012}
\end{figure}
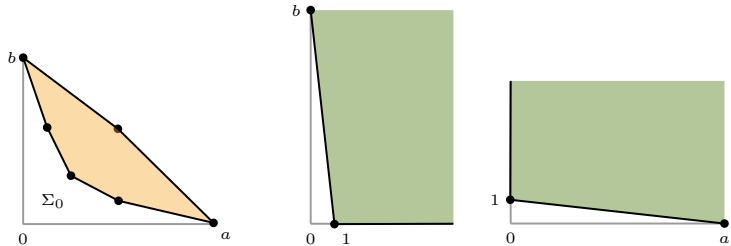

\subsubsection{Main result}\label{sub:statement} We start with the following notion.

\begin{definition}\label{def:conical}
Let $S$ be a set of five points in $\N^2\setminus\{\bm{0}\}$, and let $M(S)$ be a matrix in $\R^{5\times 5}$ with rows 
\[
(s_1,~s_2,~s_1^2,~s_1s_2,~s_2^2),
\] constructed from points $s:=(s_1,s_2)\in S$.  We say that $S$ is \emph{conical} if $M(S)$ is non-singular. A pair $(A_1,A_2)\in \ccPgeq$ is \emph{conical} if each of $A_1$ and $A_2$ contains a conical subset in $\N^2\setminus\{\bm{0}\}$. We use $\ccCgeq$ to denote the set of conical pairs in $\ccPgeq$.
\end{definition}

\begin{remark}\label{rem:conical}
The name chosen for Definition~\ref{def:conical} is justified as follows. If $S\subset\N^2\setminus\{\bm{0}\}$ is a set of five points, no three of which belong to a line, then there is a unique conic containing $S$. This conic $
1 + a_1x_1 + a_2x_2 + a_3x_1^2 + a_4x_1x_2 + a_5x_2^2$ is constructed from $S$ by solving for $\underline{a}:=(a_1,\ldots,a_5)$ the linear system 
\[
M(S)\cdot \underline{a} = -\underline{1}.
\] This shows that $\det M(S)\neq 0$. Conversely, the set $\big\{(i,j)\big\}_{1\leq i+j\leq 2}$, for example, is conical but three points belong to the line $x+y = 2$.  
\end{remark}
\begin{notation}
Let $N_*:=(\N^*)^2 $. For any subset $X\subset \R^2$, define the values $m_1(X),m_2(X)\in\N$ as follows:
\[
m_i(X):=\min\left\{\pi_i(X\cap N_*)~|~\pi_i:(x_1,x_2)\mapsto x_i\right\},~\text{for }i=1,2.
\] For any $A\in\ccPgeq$, we define $\max(0,~m_1(A_1)+ m_1(A_2)-1)$ and $\max(0,~m_2(A_1)+m_2(A_2) - 1)$ to be the \emph{vertical} and \emph{horizontal} gaps of $A$, and will be denoted by $m_v(A)$ and $m_h(A)$ respectively (see Example~\ref{ex:gaps}). We will sometimes shorten the notation by omitting the suffix $(A)$. 
\end{notation}

\begin{remark}\label{rem:gaps}
The polynomial $\det \Jac_z f$ is a product of the monomial $z_1^{m_v}\cdot z_2^{m_h}$ with a polynomial $Q\in\C[z_1,z_2]$ having non-zero constant term (i.e. $\det \Jac_z f = z_1^{m_v}\cdot z_2^{m_h}\cdot Q$ with $Q(0,0)\neq 0$).
\end{remark}

\begin{example}\label{ex:gaps}
The pair $A$ in Figure~\ref{fig:JJJAD} has a horizontal gap $2+2 - 1 = 3$, and vertical gap $\max(0,~0+0-1)=0$. For any $f\in\C^A$, we get $\det\Jac_z f$ is equal to $z_2^3\cdot Q$ for some polynomial $Q$ with a non-zero constant term.
\end{example}

\begin{notation}
If $\Pi$ is a lattice polytope in $\R^2$, we set $\hcir \Pi$ to be zero if $\dim\Pi<2$, and otherwise, to be the number of lattice points in $\Pi$ not belonging to any of its proper faces:
\[
\hcir\Pi:= \# \left(\Pi\setminus\bigcup_{\pi\precnsim\Pi} \pi\right)\cap\Z^2.
\] 
Now, if $\Pi'$ is another polytope, their \emph{Minkowski sum}, denoted by $\Pi\oplus\Pi'$, is the vector sum $\{\pi + \pi'~|~\pi\in \Pi,~\pi'\in \Pi'\}$. Their \emph{mixed volume}, denoted by $\MV(\Pi,\Pi')$, is a real-valued bi-linear, with respect to the Minkowski sum, the function of $(\Pi,\Pi')$, expressed as
\[
  \Vol(\Pi \oplus \Pi') - \Vol(\Pi) - \Vol(\Pi').
\]  We have $\MV(\Pi,\Pi')= 2\Vol(\Pi)$ whenever $\Pi = \Pi'$. 
\end{notation}
For any polynomial map $f:\CtC$, we use $\Csf$, and $\Ssf$ to denote the union of irreducible components of $\cCf$, and $\cS_f$, respectively, which are not contained in the coordinate axes. We also define $\Dsf:=f(\Csf)$. Note that each of $\Csf$, $\Dsf$, and $\Ssf$ has a convenient Newton polytope.

For any $A\in\ccPgeq$, we use $A^0$ to denote the pair $(A_1^0,A_2^0)$, where, for each $i=1,2$, the polytope $A_i^0\subset\R^2$ is the convex hull of $\{\bm{0}\}\cup A_i$.
\begin{theorem}\label{thm:discriminant}
To each conical pair $A\in\ccCgeq$, one can associate four convenient integer polytopes $\Sigma,\Delta,\Gamma,\Gamma'\subset\R^2$, and a dense subset $\Omega_A\subset\C^A$, so that for any generic $f\in\Omega_A$, the following assertions hold:
\begin{enumerate}[label=(\arabic*)]

	\item\label{it:top-multiplicity} The topological degree of $f$ is equal to $\MV(A^0)$,\\

	\item\label{it:crit_genus} the curve $\Csf$ has Newton polytope $\Sigma$, is smooth in $\C^2\setminus\{\bm{0}\}$, irreducible, has topolgical multiplicity one, satisfies $\mu_0(\Csf)=\cN(\Csf)$, and has geometric genus $\hcir\Sigma$,\\
	
		\item\label{it:crit_axes} the topological multiplicity of the horizontal and vertical lines of $\cCf\setminus\Csf$ are $m_h$ and $m_v$ respectively, and $f(\cCf\setminus\Csf)\cap\TT=\emptyset$,\\

	\item\label{it:discr_cusps} the curve $\Dsf$ has Newton polytope $\Delta$, is irreducible, has topological multiplicity one, its geometric genus equals $\hcir\Sigma$, has at most mild singularities  in $\C^2\setminus\{\bm{0}\}$, and their number is equal to
	\[
		\hcir\Sigma- \hcir\Delta,
	\] 		

	\item\label{it:non-properness} singularities of $\Ssf$ in $\C^2\setminus\{\bm{0}\}$ are simple nodes, and there are curves $\cR,\cR'\subset \C^2$ satisfying 
	\begin{align*}
		\cR\cup\cR' & =\Ssf, && & 
\mu_0(\cR\cup\cR')&=\cN(\Gamma\oplus\Gamma') && &
\#\cR\cap \cR'&= \MV(\Gamma,\Gamma')  ,\\
	\NP(\cR)  & = \Gamma, && & 
	\mu_0(\cR) & = \cN(\Gamma), && &
 \#\Sing(\cR\setminus\{\bm{0}\}) & = \hcir\Gamma,\\
 	\NP(\cR')  & = \Gamma', && & 
	\mu_0(\cR') & = \cN(\Gamma'), && &
 \#\Sing(\cR'\setminus\{\bm{0}\}) & = \hcir\Gamma'.
	\end{align*} 
\end{enumerate}
 
\end{theorem}
Theorem~\ref{thm:discriminant}  may fail for some non-conical pairs $A$. For example, for any $a,b,c,d\in\C^*$, the set $\Csf$ of the below map will have two components instead of one:
\[
f:(x,~y)\mapsto (a\cdot xy +b\cdot x^3y^3,~c\cdot xy + d\cdot x). 
\] 
%As for conical pairs, we obtain the following corollary. 

\begin{notation}\label{not:extra}
Let $\Pi,\Pi',\Pi''\subset\R^2_{\geq 0}$ be convenient polytopes. We define $\delta(\Pi)\in\{0,1\}$ to be
 \[
 \delta(\Pi)=0\Longleftrightarrow \cN(\Pi) = 0.
 \] For any operators $\cO,\cO'\in\{\hcir,\cN,\delta\}$, and operations $\circ,\circ'\in\{\cdot,+,-\}$ over $\R$, we use the abbreviations
 \begin{align*}
  \cO~(\Pi\circ\Pi'\circ \Pi'')&:=\cO(\Pi)~\circ~\cO(\Pi')~\circ~\cO(\Pi''),\\
 (\cO\circ \cO')~\Pi&:=\cO(\Pi)~\circ~\cO'(\Pi),\\
 (\cO\circ \cO')~(\Pi ~\circ'~\Pi')&:=(\cO\circ \cO')~\Pi~\circ'~(\cO\circ \cO')~\Pi'.
 \end{align*}
\end{notation}

\begin{definition}\label{def:polyhrdral-type}
Define the map
\[
\Psi:\ccCgeq\longrightarrow\N^{12},
\] taking any conical pair $A:=(A_1,A_2)\in\ccCgeq$ to the following values
\begin{align*}
\Psi_1& :=\MV(A^0) && &
\Psi_5& :=m_c\cdot m_h + m_c\cdot m_v+m_h\cdot m_v && &
\Psi_9& :=(\hcir + \cN)~(\Gamma + \Gamma') \\
\Psi_2& :=\hcir\Sigma && &
\Psi_6& :=m_c\cdot m_h\cdot m_v && &
\Psi_{10}&:=(\hcir + \cN)~(\Gamma\cdot\Gamma') \\
\Psi_3& :=\cN(\Sigma) && & 
\Phi_{7}& :=\hcir(\Delta - \Sigma) + \delta(\Delta) && &
\Psi_{11}& :=(\hcir + \delta)~(\Gamma + \Gamma')\\ 
 \Psi_4 &:= m_c+m_h+m_v && &
\Psi_8& :=\MV(\Gamma,\Gamma') + \cN(\Gamma\oplus\Gamma'-\Gamma-\Gamma')&& &
 \Psi_{12}& :=(\hcir + \delta)~(\Gamma \cdot\Gamma'),
\end{align*} where $m_c\in\{0,1\}$ satisfies $m_c=0\Longleftrightarrow \dim\Sigma <1$. For any generic $f\in\C^A$, the image $\Psi(A)\in\N^{12}$ is called the \emph{polyhedral type of $f$}.
%The image $\Psi(A)$ is called the \emph{restricted type} of $A$.
\end{definition}

\begin{example}\label{ex:main}
%All the computations in this example are represented in [REF].
If $A\in\ccCgeq$ is the conical pair illustrated in Figure~\ref{fig:pair-example}, then, we have $\MV(A^0)=20$, $m_h = 3$ and $m_v=0$. Let $\C^A\ni f:=(f_1,f_2):\C^{2}_{x,y}\to\C^2_{a,b}$ be given by 
\begin{align*}
f_1&=  ~2~x^2y^2 - ~y^2 +~ x^2y^6 -~ 3~x^4y^4 +~ x^3y^5,\\
f_2&= ~x^2y^2 - ~5~xy^2 + ~7~x^3y^6 -~11~x^5y^5 + ~21~x^4y^5 +~x^2y^4.
\end{align*} Its critical locus $\cCf$ is the union $\{y=0\}\cup\Csf$, and $\Sigma$ (represented in Figure~\ref{fig:JJJAD}) can be computed by hand. The discriminant $\cDf$ can be obtained via elimination methods, and thus the vertices of $\Delta$ are
\[
\begin{blockarray}{cccccc}
\begin{block}{(cccccc)}
  0 & 0 & 4 & 15 & 30 & 30 \\
  4 & 32 & 0 & 20 & 0 & 6\\
\end{block}
\end{blockarray}.
 \] This yields $\hcir\Sigma = 19$, $\hcir\Delta = 544$, $m_c=1$, $\cN(\Sigma) = 0$, and $\cN(\Delta) = 9$. The non-properness set $\cS_f$ is given by the union of two horizontal lines $\cR:=\{(b-1)\}\cup\{(b+11)\}$, together with $\cR'$, given by
{\small 
\begin{align*}
P:= &~14641~a^5 + ~4356~a^3b -~ 10890~a^2b^2 + ~243~b^4 + ~2178~a^3\\
\ & - ~11616~a^2b + 14682~ab^2 - ~324~b^3 - ~941~a^2 +~ 3764~ab - ~3764~b^2.
\end{align*}
} The computations were done using forward Theorem~\ref{thm:non-prop-main}. Then, we get $\Gamma$ is a vertical line segment, and $\Gamma'$ is presented in Figure~\ref{fig:JJJAD}. We can thus compute $\MV(\Gamma,\Gamma') = 10$, $\cN(\Gamma)=\hcir(\Gamma)=0$, $\cN(\Gamma')=\cN(\Gamma\oplus\Gamma')=1$, $\hcir(\Gamma') = 5$. Altogether, this gives the following polyhedral type
\[
\Psi(A) =  (20,~19,~0,~4,~3,~0,~526,~10,~6,~0,~6,~0).
\] 
\end{example}

\begin{figure}[htb]
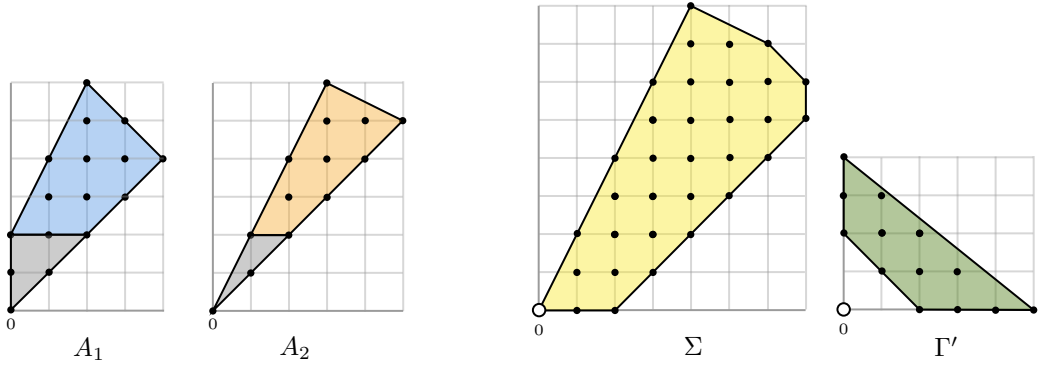


\tikzset{every picture/.style={line width=0.75pt}} %set default line width to 0.75pt        



\caption{Left: The pairs $A$ (blue and orange) and $A^0$ (grey included). Right: The polytopes $\Sigma$ and $\Gamma'$ corresponding to Example~\ref{ex:main}.}\label{fig:JJJAD}
\end{figure}
\begin{theorem}\label{thm:main2}
Let $\Upsilon:\cMCt\to\N^I$ be the map in Corollary~\ref{cor:main}. Then,
exists a map $F:\N^I\to\N^{12}$, such that for any $A\in\ccCgeq$, and any generic $f\in\C^A$, the following equality holds
\[
F\circ \Upsilon(f) = \Psi(A).
\] In particular, the polyhedral type of $f$ is a topological invariant.
\end{theorem}
\begin{proof}
Let $f\in\C^A$ be a generic map. In what follows, for each $i=1,\ldots,12$, we show that $\Psi_i(A)$ is a topological invariant of $f$. For $i=1,2,3$, the result follows immediately from Theorem~\ref{thm:discriminant}~\ref{it:top-multiplicity} and \ref{it:crit_genus}.

Theorem~\ref{thm:discriminant}~\ref{it:crit_genus} and~\ref{it:crit_axes}  show that $\cCf$ has at most three components whose topological multiplicities form a triple $(m_c,m_h,m_v)$. Indeed, the topological multiplicity of $\Csf$ is one if $\Csf$ is non-empty, and $\Sigma$ is either a point or empty otherwise. As homeomorphisms $\CtC$ may permute irreducible components, we consider triples $(\Psi_4,\Psi_5,\Psi_6)$ instead. The curve $\cDf$ also has at most three components, among of which $\Dsf$ is the only one that can be singular. Thanks to Lemma~\ref{lem:Milnor=0} and Theorem~\ref{thm:discriminant}~\ref{it:discr_cusps}, the value $\Psi_7$ counts the total number of singular points of $\cDf$ in $\C^2$. This is a topological invariant according to Proposition~\ref{prp:separator}~\ref{it:jacobian-multiplicity}. 

The remaining values $\Psi_i$ ($i=8,\ldots,12$) encode the singularities of $\Ssf$ in $\C^2$. The latter are divided into three subsets $\Sing(\cR)$, $\Sing(\cR')$, and $\cR\cap \cR'$. Furthermore, thanks to Theorem~\ref{thm:discriminant}~\ref{it:non-properness}, we have the following identities
\begin{align*}\label{eq:identities2}
\#\Sing (\cR) & = (\hcir +\delta)~\Gamma,  &
\#\Sing (\cR') & = (\hcir +\delta)~\Gamma', 
\\
\sum\mu_p(\cR) & = (\hcir +\cN)~\Gamma &
\sum\mu_p(\cR') & = (\hcir +\cN)~\Gamma'
\end{align*} where the sum runs over all singular points of the respective curves $\cR$ and $\cR'$. Proposition~\ref{prp:separator}~\ref{it:jacobian-multiplicity} shows that the values above are topological invariants of $f$ up to an index permutation $1\leftrightarrow 2$. To quotient the latter, we instead use the values $\Psi_9,\Psi_{10},\Psi_{11},\Psi_{12}$. Similarly for $\Ssf$, the sum $\sum \mu_p(\Ssf)$ is also a topological invariant of $f$ and is equal to
\[
\hcir\Gamma+\hcir\Gamma'+\MV(\Gamma,\Gamma') +\mu_0(\Ssf)
\] according to Theorem~\ref{thm:discriminant}~\ref{it:non-properness}. Note that $\mu_0(\Ssf) = \cN(\Gamma\oplus\Gamma')$. Then, since $(\hcir+ \cN)~\Gamma$ and $(\hcir+ \cN)~\Gamma'$ are accounted for in $\Psi_9$ and $\Psi_{10}$, the value $\Psi_8$ is a topological invariant on $f$.
\end{proof}

\begin{remark}\label{rem:cumbersome}
In order to compute the full map $\Upsilon$, Theorem~\ref{thm:main2} is missing the value $\mu_0(\cDf)$, the exact number of ordinary cusps of $\cDf$, and the number of irreducible components of $\cD_f$ and $\cS_f$, together with their topological multiplicities. We have omitted to compute these values as the statements would become significantly more cumbersome. Additionally, Proposition~\ref{prp:separator} easily extends to other identities involving the fundamental group $\pi_1(\C^2\setminus X)$, for $X\in\{\cCf,\cDf,\cS_f\}$. Thus, so would the map $\Upsilon$. Computing those invariants for conical pairs is the topic of future work.
\end{remark}

\section{Affine singularities of the discriminant}\label{sec:affine-sing}  
This section is devoted to describing the singularities of the discriminant corresponding to generic maps. The methods are adapted from~\cite{golubitsky2012stable} and~\cite{FJR19}. Let $A\in\ccCgeq$ be a conical pair, and consider the space $\C^{A^0}$ of maps $f\in\C^A$ after adding arbitrary constant terms to each polynomial. For any $f\in\C^{A^0}$, and any $a\in\C^2$, we set $\partial f(a)$ and $\partial^2 f(a)$ to denote the first two tuples of partial derivatives at $a$:
\[
\left.\left(\frac{\partial f_i}{\partial z_j}\right|_{z=a} \right)_{1\leq i,j\leq 2}\text{ and }\left.\left(\frac{\partial^2 f_k}{\partial z_i\partial z_j}\right|_{z=a} \right)_{1\leq k\leq 2,~1\leq i\leq j\leq 2}.
\]
Altogether, points $(a,~f(a),~\partial f(a),~\partial^2 f(a))$ form the \emph{space of $2$-jets} $\cJ^2(\TT,~\C^2)=:\cJ^2$ containing the image of the algebraic \emph{Thom-Boardman map of $A$} (c.f.~\cite{mather1973generic})
\begin{align*}
\Phi:\TT\times\C^{A^0}&\longrightarrow \TT\times \C^{12}\\
(z,~f)& \longmapsto (z,~f(z),~\partial f(z),~\partial^2 f(z)),
\end{align*} 

\begin{lemma}\label{lem:submersion}
The Thom-Boardman map of $A$ is a submersion.
\end{lemma}

\begin{proof}
We follow closely the proof of~\cite[Theorem 2.5]{FJR19}. Assume without loss of generality that $\#A_1 = \#A_2 = 5$. We will show that the map 
\begin{align*}
\Phi_a:=\Phi(a,\cdot):\C^{A^0}&\longrightarrow \C^{12}\\
f& \longmapsto (f(a),~\partial f(a),~\partial^2 f(a))
\end{align*} is a submersion for any $a\in\TT$. Since $\Phi_a$ is linear, it is enough to show that it is surjective. Let $i\in\{1,2\}$, and consider $\Phi_a^i:\C^{A_i^0}\to\C^6$, $h\mapsto (h(a),~\partial h(a),~\partial^2 h(a))$, taking one bivariate polynomial to its partial derivatives of order at most two and evaluating at $a$. Since $\Phi_a^1,\Phi_a^2$ is a splitting of $\Phi_a$, we only show that $\Phi_a^i$ is surjective. The latter is given by the matrix $M\in\N^{6\times 6}$, whose first row is the vector $(1,0,\ldots,0)$ and each of the remaining five rows is expressed as 
\[
(1,~r,~s,~r\cdot (s - 1),~r\cdot s,~s\cdot (s - 1)),
\] for the corresponding point $(r,s)\in A_i$. After some linear operations on $M$, we obtain the matrix with rows $(1,0,\ldots,0)$ and
\[
(1,~r,~s,~r^2,~r\cdot s,~s^2).
\] As $A_i$ is conical, Definition~\ref{def:conical} implies that $\det M\neq 0$, and thus $\Phi^i_a$ is surjective.
\end{proof} Denote the coordinates of $\cJ^2$ as 
\[
(z_1,z_2,u,v,u_1,u_2,v_1,v_2,u_{11},u_{12},u_{22},v_{11},v_{12},v_{22}),
\] where $(u,v)$ corresponds to $(f_1,f_2)$, and the indexes correspond  partial derivatives of the latter. The Thom-Boardman map is useful in describing the critical locus, the discriminant of maps $f\in\C^A$, and its singularities: If $\Sigma\subset \cJ^2$ denotes the hypersurface $\{u_1v_2-u_2v_1 = 0\}$, then $\Csf$ and $\Dsf$ can be expressed as the closures of $
 \pi_1 (\Lambda_f\cap\Sigma)$ and $\pi_2 (\Lambda_f\cap\Sigma)$, where
\[
\Lambda_f:=\Phi(\TT,f),
\] and $\pi_1$ and $\pi_2$ are the projections $\cJ^2\to\TT_{z}$ and $\cJ^2\to\C^2_{u,v}$ respectively. 
It is easy to check that $f$ has a cusp at $a$ if an only if $\cDf$ has an ordinary cusp at $f(a)$ in the sense of Remark~\ref{rem:cusp-node_Milnor}. Let $\Sigma^1$ 
denote the hypersurface in $\cJ^2$, whose equation is 
\[
(u_{11}v_{2}+u_{1} v_{12}-u_{12}v_1 - u_{2}v_{11} )\cdot u_{2} - ( u_{12}v_{2} + u_{1} v_{22} - u_{22}v_{1} - u_{2}v_{12})\cdot  u_{1} =  ~0
\]
The following result is well-known thanks to Whitney~\cite{whitney1955singularities} (see also~\cite[pp 1--2]{levine1965elimination} and~\cite[Remark 3.3]{FJR19}). 
\begin{lemma}\label{lem:Whitney-singularities}
For any $f\in\C^{A}$, if $\Phi(a,f)\in\Sigma^1\cap\Sigma$, and $\Lambda_f$ intersects transversally each of $\Sigma^1$ and $\Sigma$ at $\Phi(a,f)$, then $a\in\C^2$ is an ordinary cusp of $f$.
\end{lemma}  Before we prove the main theorem of this section, we introduce the following notion.

\begin{definition}[\cite{golubitsky2012stable}, Definition 3.5,~pp 83]\label{def:general_position}
Let $V$ be a vector space and let $H_1,\ldots,H_r$ be subspaces
of $V$. Then $H_1,\ldots,H_r$ are said to be in general position if for every sequence of integers $i_1,\ldots,i_s$ with $1\leq i_1<\ldots < i_s\leq r$, we have 
\[
\cdim (H_{i_1}\cap\cdots\cap H_{i_s}) = \cdim H_{i_1}+\cdots+ \cdim H_{i_s}.
\] 
\end{definition}

\begin{remark}\label{rem:two_general_positions}
One can check that two subspaces $H_1,H_2\subset V$ are in general position iff $H_1 + H_2 = V$ (see also \cite[Definition 3.5,~pp 83]{golubitsky2012stable}).
\end{remark}

\begin{theorem}\label{thm:critical-smooth}
Let $A\in\ccCgeq$. Then, there exists a dense subset $U_c\subset\C^{A}$ in which each $f\in U_c$ satisfies the following assertions:
\begin{enumerate}[label=(\arabic*)]

\item\label{it:smoothness} the curve $\Csf$ is smooth in $\C^2\setminus\{\bm{0}\}$, has topological multiplicity one, and is irreducible,

\item\label{it:genus=genus} the geometric genus of $\Csf$ equals the geometric genus of $\Dsf$, and

\item\label{it:cusps-nodes} a singular point of $\Dsf$ in $\C^2\setminus\{\bm{0}\}$ is either a simple node or an ordinary cusp.

\end{enumerate}
\end{theorem}

\begin{proof} 
Lemma~\ref{lem:submersion} shows that there exists an open dense subset $V\subset\C^{A^0}$ containing all elements $f\in V$ for which $\Lambda_f$ is transveral to $\Sigma$. Define $U_c:=\Proj(V)$ to be the image under the projection $\C^{A^0}\to\C^A$, $f\mapsto f-f(\bm{0})$. Then, for any $f\in U_c$, we have $\Csf$ has no singularities in $\TT$, has topological multiplicity one, and is irreducible. Indeed, since $\Sigma$ is irreducible and has multiplicity one (as a sheaf). Now, to see that $\Csf$ has no singularities at the coordinate axes outside $\{\bm{0}\}$, note that $\Sigma$ intersects transversally $\{u=0\}$ and $\{v=0\}$ outside the linear subspace $\{uv=0\}\subset\cJ^2$. This proves Item~\ref{it:smoothness}.

Concerning Items~\ref{it:genus=genus} and~\ref{it:cusps-nodes}, we we start with some observations. It is known~\cite[Proposition 2.1, pp. 78]{golubitsky2012stable} the map $\Phi$ is stable as a consequence of it being a submersion. We refer to stability in the classical sense introduced by Whitney (see e.g.~\cite[Definition 1.1, pp. 72]{golubitsky2012stable}). Then, the dense subset $V\subset \C^{A^0}$ can be chosen so that for any $f\in V$, the restricted map $\Phi_f$ is also stable, where 
\begin{align*}
\Phi_f:\TT&\longrightarrow Y\\
z& \longmapsto \Phi_f(z):=\Phi(z,f),
\end{align*} and $Y$ is the closure of $\Lambda_f$ in $\cJ$. Furthermore, recall that $\Lambda_f$ intersects transversely the set $\Sigma$. Then, a straightforward computation shows that $\Phi_f$ is an immersion. 

Now, let $b\in\Dsf\setminus\{\bm{0}\}$. Since $\bm{0}\not\in A_1\cup A_2$, the preimage $f^{-1}(b)$ does not contain the origin. Consider any finite collection of points $\{a_1,\ldots,a_n\}\subset f^{-1}(b)$. Since $\Phi_f$ is a stable immersion,~\cite[Theorem 3.11, pp. 85]{golubitsky2012stable} together with~\cite[Lemma 3.6, pp. 83]{golubitsky2012stable} show that the subspaces 
\[
(d\Phi_f)_{a_1}(T_{a_1}\TT),~\ldots,~(d\Phi_f)_{a_n}(T_{a_n}\TT)~\subset T~ Y
\] are in general position. Using again the transversality of $\Lambda_f\cap\Sigma$, we conclude that the lines $H_1,\ldots,H_n\subset T~\Phi_f(\Csf)$, given by 
\[
H_i:=\left.\left(d\Phi_f \right|_{\Csf}\right)(T_{a_i}\Csf)=T_{\Phi_f(a_i)}\Sigma\cap T_{\Phi_f(a_i)}Y,
\] are in general position in the tangent bundle $T~\Sigma\cap Y$. This shows that so are the lines
\[
(df)_{a_1}(T_{a_1}\Csf),~\ldots,~(df)_{a_n}(T_{a_n}\Csf)\subset T~\Dsf, 
\] in general position. Since $\dim T~\Dsf = 2$, and $\dim H_1 =\cdots= \dim H_n =1$, Remark~\ref{rem:two_general_positions} shows that $n\leq 2$. In fact, this yields the following three assertions (c.f.~\cite[Lemma 4.1]{FJR19})
\begin{itemize}
\item[(i)]\label{it:subit-inj} $f_{|\Csf}$ is injective outside a finite subset,

\item[(ii)]\label{it:subit-preimag-number} if $b\in\Dsf\setminus \{\bm{0}\}$ then $\# f^{-1}(b)\cap\Csf~\leq 2 $, and 

\item[(iii)]\label{it:subit-node} if $\# f^{-1}(b)\cap\Csf~= 2 $ then the curve $\Dsf$ has a simple node at $b$.
\end{itemize} Finally, Item~\ref{it:genus=genus} follows from Item~(i), and Item~\ref{it:cusps-nodes} follows from Items (ii), (iii), and Lemma~\ref{lem:Whitney-singularities}.
\end{proof}

\section{Boundary singularities of the discriminant}\label{sec:toric_sing}
In this section, we describe a compactification of the discriminant of a generic map $\CtC$ based on its Newton polytope. We then show in Theorem~\ref{thm:discr-infty} below that it does not introduce new singularities. This allows us to use classical theorems from toric geometry to describe the genus and singularities of the discriminant. The proof of forward Theorem~\ref{thm:discr-infty} requires some preparatory results and prerequisites in toric geometry. These will be presented in~\S\ref{sub:faces-polytopes} and~\S\ref{sub:toric} respectively. 

\subsection{Preliminaries on Newton polytopes of critical loci}\label{sub:faces-polytopes}  

The common zero locus in $\C^n$ of a tuple $P$ of polynomials $P_1,\ldots,P_r\in \C[z_1,\ldots,z_n]$ will be denoted by $\V(P):=\V(P_1,\ldots,P_r)$, and define $\Vs(P):=\V(P)\cap \TTn$. 
If $\Pi$ is the Newton polytope of a polynomial $P=\sum c_a~z^a\in\C[z_1,\ldots,z_n]$, and $\delta\prec\Pi$ is a face of $\Pi$, then we use $P_\delta$ to denote the polynomial
\[
\sum_{a\in \delta\cap\N^n}c_a~z^a.
\] We say that $P$ is \emph{non-degenerate at $\delta$} if 
\[
\Vs(P_{\delta},~\partial P_{\delta}/\partial z_1,\cdots,~\partial P_{\delta}/\partial z_n)=\emptyset, 
\] and \emph{Newton non-degenerate} if it is non-degenerate at each face of $\Pi$.
\begin{lemma}\label{lem:Newton_polytopes-invar}
For  any conical pair $A\in\ccCgeq$, there exists lattice polytopes $\sJ$, $\sDc$, and a dense subset $U_1\subset\C^A$ in which each $f\in U_1$ satisfies $\sJ =  \NP(\Csf)$ and $\sDc = \NP(\Dsf)$. 
\end{lemma}
\begin{proof}
We will define $U_1$ to be the intersection of Zariski opens $K,L\subset \C^A$, where each one represents maps $f$ satisfying one of the two above identities.

Concerning $\sJ$. Every $(f_1,f_2)\in\C^A$ is obtained from a pair of polynomials $F_1,F_2\in\Z[z_1,z_2,c_u:u\in A_1]\times\Z[z_1,z_2,d_v:v\in A_2]$ after substituting the corresponding values for the coefficients 
\[
\underline{c}:=\big(c_u\big)_{u\in A_1}\text{\quad and\quad }\underline{d}:=\big(d_v\big)_{v\in A_2}.
\] Then, we can express $\det\Jac_z F$ as the polynomial
$\sum\varphi_s~z^s$ 
with support $S\subset\N^2$, where $\varphi_s\in\Z[\underline{c},\underline{d}]$ are of the form 
\begin{equation}\label{eq:phi-s}
\sum_{\substack{u\in A_1, v\in A_2 \\ 
u+v = (s_1+1,s_2+1)}}c_ud_v\cdot\det(u,v)
\end{equation} Hence, the equality $\conv(S)=\NP(\det\Jac_z f)=:\Sigma$ holds if the coefficients of $f\in\C^A$ satisfy $\prod\varphi_s\neq 0$. All such choices of $(\underline{c},\underline{d})$ form a Zariski open $K$ containing $\C^A\setminus\V(\prod\varphi_s)$.

As for $\sDc$, we follow similar steps as above; the elimination ideal 
\[
\langle F_1-w_1,~F_2 - w_2,~\det\Jac_z F\rangle\cap \Z[w_1,~w_2,~\underline{c},\underline{d}],
\] is generated by $\mathsf{D}_F:=\sum\psi_t~w^t$, supported on a finite set $T$, where $\psi_t\in \Z[\underline{c},\underline{d}]$. The polynomial $\mathsf{D}_F$ specializes to the equation of $\cDf$ when plugging $(\underline{c},\underline{d})$ corresponding to $f$. Similarly, as above, we take the dense subset $L:=\C^A\setminus\V(\prod\psi_t)$, and $\Delta:=\conv(T)$.
\end{proof}

We will present in~\S\ref{sec:algorithms} a method for computing $\Delta$ in Lemma~\ref{lem:Newton_polytopes-invar} that uses only $A$.

\begin{lemma}\label{lem:C_Newton-non} 
There is a dense subset $U_2\subset\C^A$ in which each $f\in U_2$ is a polynomial map $\CtC$ with Newton non-degenerate critical locus $\Csf$.
\end{lemma}

\begin{proof}
 Let $\sigma$ be any edge of $\Sigma$. It is enough to construct a dense subset $U_\sigma\subset\C^A$ in which the set $\Vs(Q_{\sigma}, \partial Q_{\sigma}/\partial z_1,\partial  Q_{\sigma}/\partial z_2)$ is empty, where $Q:=\det(\Jac_z f)$. Then, we take $U_2$ to be the locus of $U_1$ from Lemma~\ref{lem:Newton_polytopes-invar}, intersected with all $U_\sigma$ corresponding to faces of $\Sigma$.

Let $v$ be an integer vector spanning $\sigma$. Then, the polynomial $Q_\sigma$ can be written as $z^u\cdot h(z^v)$ for some $u\in\N^2$, where $h$ is a univariate polynomial with a non-zero constant term. Furthermore, it is an easy computation to check that $a\in \Vs(Q_{\sigma}, \partial Q_{\sigma}/\partial z_1,\partial  Q_{\sigma}/\partial z_2)$ if and only if $a^v = c_0$ for some solution $c_0\in\C^*$ to $h = h' = 0 $. Now, consider the \emph{$A$-discriminant} $\nabla_B$ in the space $\C^{B}\cong\C^{\#\sigma \cap\N^2}$ of all univariate polynomials with the same support $B\subset \N$ as $P$. This is the closure of the set of all such polynomials having a double root in $\C^*$. On the one had, if non-empty, the $A$-discriminant $\nabla_B$ is an algebraic hypersurface in $\C^B$ (see e.g.,~\cite{GKZ08}). On the other hand, we leave it to the reader to check the map $\Upsilon:\C^A\longrightarrow\C^{\sigma}\cong\C^B$, $f\longmapsto (\det \Jac_z(f))_\sigma$ is dominant. Then, the set $U_\sigma:=\Upsilon^{-1}(\C^B\setminus\nabla_B)$ is a dense in $\C^A$. 
\end{proof}

\subsection{Preliminaries on toric varieties}\label{sub:toric}
We give here a brief introduction to toric varieties following closely~\cite{Ful93} and~\cite{cox2011toric}. 

\begin{definition}\label{def:toric_variety}
A \emph{toric variety} $X$ is a normal variety that contains a torus $\TTn$ as a dense open subset, together with an action $\TTn\times X\to X$ of $\TTn$ on $X$ extending the natural action of $\TTn$ on itself.
\end{definition}

For any lattice $N\cong\Z^n$ we write $N_\R:=N\otimes_\Z\R$. A \emph{rational cone} $\sigma\subset N_\R$ is given as a conical combination of finitely many vectors $a_1,\ldots,a_r\in N$: 
\[
\sigma:=\{\lambda_1a_1+\cdots+\lambda_ra_r~|~\lambda_1,\ldots,\lambda_r\in\R_{\geq 0}\}.
\] This is a convex set, which is called \emph{strongly convex} if $\sigma$ does not contain a positive-dimensional subspace of $N_\R$. 
%no line passing through the origin contains $\sigma$.  
A \emph{face} $\tau$ of $\sigma$, expressed as $\tau\prec\sigma$, is the intersection of $\sigma$ with a hyperplane $H$ passing through the origin in which $\sigma$ is entirely contained in one of the two halfspaces determined by $H$.

\begin{definition}\label{def:fans}
A \emph{fan} $\Gamma$ in $N$ is a collection of strongly convex rational cones in $N_\R$, where every face of a cone in $\Gamma$ is also a cone in $\gamma$, and the intersection of two cones in $\Gamma$ is a face of each. The last two properties make $\Gamma$ \emph{polyhedral}.
\end{definition}

We use $M:=\Hom(N,\Z)$ to denote the dual lattice with dual pairing denoted by $ \langle\cdot,\cdot\rangle$. If $\sigma$ is a cone in $N_\R$, the dual cone $\sigma^{\vee}$ is the set of vectors in $M_{\R}:=M\otimes_\Z\R$ that are nonnegative on $\sigma$. This determines a commutative semigroup
\[
S_\sigma = \sigma^{\vee}\cap M = \{u\in M~|~\langle u,~v\rangle\geq 0~\text{for all }v\in\sigma\}.
\] This semigroup is finitely generated, so its corresponding \emph{group
algebra} $\C[S_\sigma]$ is a finitely generated commutative $\C$-algebra. Such an algebra corresponds to an affine variety by setting
\[
U_\sigma := \Spec(\C[S_\sigma]).
\] If $\tau$ is a face of $\sigma$, then $S_\sigma$ is contained in $S_\tau$, so $\C[S_\sigma]$ is a subalgebra of $\C[S_\tau]$, which gives a map $U_\tau\to U_\sigma$. In fact, $U_\tau$ is a principal open subset of $U_\sigma$: if we choose $u\in S_\sigma$ so that $\tau = \sigma\cap u^{\perp}$, then $U_\tau = \{x\in U_\sigma~|~u(x)\neq 0 \}$.

\begin{theorem}[\cite{Ful93}]\label{thm:toric_1}
Let $\Gamma$ be a fan in a lattice $N$, and let $X_\Gamma$ be the algebraic variety obtained as a result of gluing all affine varieties $U_\sigma$, where $\sigma$ runs over all cones of $\Gamma$. Then, $X_\Gamma$ is a toric variety. 
\end{theorem}

\subsubsection{Orbits of a toric variety} The toric variety from Theorem~\ref{thm:toric_1} will be called the \emph{toric variety correponding to $\Gamma$}. To the lattice $N\cong\Z^n$, we associate the torus $T_N\cong(\C^*)^n$. For each $\tau\prec\sigma$ we define $N_\tau$ to be the sublattice of $N$ generated (as a group) by $\tau\cap N$, and 
\[
N(\tau) :=  N/N_\tau,\quad M(\tau):= \tau^{\perp}\cap M,
\] the quotient lattice and its dual respectively. Define $O_\tau$ to be the torus corresponding to these lattices:
\[
  T_{N(\tau)} = \Hom(M(\tau),\C^*) = \Spec(\C[M(\tau)]) =N(\tau)\otimes_{\Z}(\C^*).
\] This is a torus of dimension $n-k$, where $k = \dim(\tau)$, on which $T_N$ acts transitively via the projection $T_N\to T_{N(\tau)}$. There is a convenient description of the torus orbits in terms of the corresponding fan. 

\begin{theorem}[\cite{Ful93},~Propostion 3.1]\label{thm:toric_2}
There are the following relations among orbits $O_\tau$, orbit closures $V(\tau)$, and the affine open sets $U_\sigma$:
\begin{equation}
U_\sigma  = \bigsqcup_{\tau\prec\sigma} O_\tau,\quad V(\tau)  = \bigsqcup_{\gamma\succ\tau} O_\gamma,\quad \text{and\quad}O_\tau  = V(\tau)\setminus \bigcup_{\gamma\underset{\neq}{\succ}\tau} V(\gamma).
\end{equation}
\end{theorem}
We note that the map $\tau \mapsto V(\tau)$ defines an inclusion reversing bijection between cones $\tau \in \Sigma$ and torus invariant subvarieties of $X_\Sigma.$
\subsubsection{Toric morphisms}\label{sss:toric_morph} Let $N_1$, $N_2$ be two lattices with $\Sigma_1$ a fan in $N_{1}$ and $\Sigma_2$ a fan in $N_{2}$. A $\Z$-linear map $\overline{\phi}:N_1\to N_2$ is \emph{compatible} with the fans $\Sigma_1$ and $\Sigma_2$ if for every cone $\sigma_1\in\Sigma_1$, there exists a cone $\sigma_2\in\Sigma_2$ such that $\overline{\phi}_\R(\sigma_1)\subset\sigma_2$. 

It is known (see e.g.,~\cite[Theorem 3.3.4]{cox2011toric}) that such linear maps between lattices induce \emph{toric morphisms} between the corresponding toric varieties and vice-versa. That is, a morphism $\varphi:X_{\Sigma_1}\to X_{\Sigma_2}$ mapping the torus $T_{N_1}\subset X_{\Sigma_1}$ into $T_{N_2}\subset X_{\Sigma_2}$ and such that $\varphi_{|T_{N_1}}$ is a group homomorphism.  

\begin{theorem}[\cite{cox2011toric},~Lemma 3.3.21]\label{thm:toric_morphisms}
Let $\varphi:X_\Sigma \to X_{\Sigma'}$ be the toric morphism coming from a map $\overline{\phi}:N\to N'$ that is compatible with $\Sigma$ and $\Sigma'$. Given $\sigma\in\Sigma$, let $\sigma'\in\Sigma'$  be the minimal cone of $\Sigma'$ containing $\overline{\phi}_\R(\sigma)$. Then, the induced map $\phi_{|V(\sigma)}:V(\sigma)\to V(\sigma')$ is a toric morphism.
\end{theorem}

\subsubsection{Toric varieties from lattice polytopes}\label{sss:toric-lattice} Given a torus $\TTn$, a set $A= \{\alpha_1,\ldots ,\alpha_s\}\subset \Z^n$ gives rise to group homomorphisms $\chi_i:\TTn\to\C^*$, $z:=(z_1,\ldots,z_n)\mapsto z^{\alpha_i}$, ($i=1,\ldots,s$) called \emph{characters}. Consider the map $\Phi_A:\TTn\longrightarrow (\C^*)^s$, given by
\begin{equation}\label{eq:torus-map}
z:=(z_1,\ldots,z_n)\longmapsto (\chi_1(z),\ldots,\chi_s(z))
\end{equation} 
If $\varsigma:(\C^*)^n\longrightarrow\P^{s-1}$ is the map
\[
(y_1,\ldots,y_s)\longmapsto [y_1/y_s:\cdots:y_{s-1}/y_s:1],
\] then the closure of the image of $\TTn$ under $\varsigma\circ \Phi_A$ is a toric variety in $\P^{s-1}$. We will denote it by $X_A$ or $X_{\conv(A)}$.

Let $\Delta\subset\R^n$ be a full-dimensional lattice polytope, and let $\gamma\prec\Delta$. We denote $\sigma_\gamma\subset \R^n$ to be the \emph{normal cone} formed by all supporting vectors of the face $\gamma$ (see~\S\ref{sss:not-prelim}). The fan $\Sigma_\Delta$, formed by all normal cones to faces of $\Delta$ is called the \emph{normal fan} of $\Delta$:
\[
\Sigma_\Delta:=\bigcup_{\gamma\prec\Delta}\sigma_\gamma
\]

%\begin{lemma}[\cite{cox2011toric},~Theorem 3.3.4]\label{lem:fan-polytope}
%We have the following properties between $\Delta$ and its normal fan $\Gamma_\Delta$:
%\begin{enumerate}[label=(\Alph*)]
%
%	\item for all $\gamma_\delta\in\Gamma_\Delta$, each face of $\gamma_\delta$ is also in $\Gamma_\Delta$, and
%	
%	\item the intersection $\gamma_\delta\cap\gamma_{\delta'}$ of any two cones in $\Gamma_\Delta$ is a face of each.
%
%\end{enumerate}
%
%\end{lemma}
%
%\begin{theorem}[\cite{cox2011toric}, Proposition 3.1.6]\label{thm:toric-correspondence}
%Let $P\subset\R^n$ be a full-dimensional lattice polytope, and let $\Gamma$ be its normal fan. Then the projective toric variety $X_{\Delta}$ is isomorphic to $X_{\Gamma}$.
%\end{theorem}

\begin{remark}\label{rem:faces-cones}
The toric varieties $X_{\Delta}$ and $X_{\Sigma_\Delta}$ need not be equal. It is however true that $X_{\Sigma_\Delta}$ is the normalization of $X_{\Delta}$.
Since the normalization map $X_{\Sigma_\Delta} \longrightarrow X_{\Delta}$ is torus equivariant, Theorem \ref{thm:toric_2} is also true for $X_{\Delta}$.
In particular, the orbits of $X_\Delta$ are in inclusion preserving bijection to the faces of $\Delta$.
\end{remark}

We end this part with the following classical correspondence.

\begin{theorem}[\cite{Ful93}, Section 1.4]\label{thm:toric_product}
Let $\Delta_1,\Delta_2\subset\R^n$ be two full-dimensional lattice polytopes. Then, it holds that
\[
X_{\Delta_1\times\Delta_2}\cong X_{\Delta_1}\times X_{\Delta_2}.
\]
\end{theorem}

\subsection{Singularities of the discriminant at the orbits}\label{sub:discr-orbits} For every set $H\subset\TTn$, we use $\overline{H}$ to denote the closure of its image in $X_\Delta$ under the map $\varsigma\circ \Phi_{\Delta}$. We call $\overline{H}$ the \emph{compactification of $H$ inside $X_\Delta$}. This compactification is said to be \emph{canonical} if $H$ is a hypersurface whose Newton polytope is $\Delta$. In this case, $\overline{H}$ does not intersect the zero-dimensional orbits of $X_\Delta$ (see e.g.,~\cite{khovanskii1978newton}). 
Then, we use $\partial H$ to denote the set $\overline{H}\setminus H\subset X_\Delta$.

\begin{remark}\label{rem:singular-orbit-Newton-non}
Note that if the canonical compactification $\overline{H}\subset X_\Delta$ has a singular point in $\partial H$, there is a non-transversal intersection of $\overline{H}$ with an orbit $V(\gamma)\cong\P^{\dim\gamma}$ of $X_\Delta$ for some $\gamma\prec\Delta$ (see Remark \ref{rem:faces-cones}). Thus, the polynomial $P$, defining $H$ is degenerate at $\gamma$. We conclude that if $H$ is Newton non-degenerate, then $\overline{H}$ is smooth at $\partial H$.
\end{remark}

\begin{theorem}\label{thm:discr-infty}
Let $A\in \ccCgeq$. Then, there is a dense subset $U_d\subset\C^A$, such that for any $f\in U_d$, the discriminant in the torus, $D:=\Dsf\cap\TT$, has only smooth points at the boundary of its compactification $\ovD\subset X_{\Delta}$:
\[
\Sin(\ovD)\cap \partial D =\emptyset.
\]
\end{theorem}

\begin{proof}
We consider an element $f\in U_1\cap U_2$, in the common open sets from 
Lemmas~\ref{lem:Newton_polytopes-invar}, and~\ref{lem:C_Newton-non}. This proof is divided into three parts: \textbf{Part 1.} describes how $\ovD$ and its singularities at the orbits are obtained from a toric morphism on $\overline{C}:=\overline{\Csf}$. We furthermore show that $\ovD$ can only have cusps or nodes as singularities. \textbf{Part 2.} is devoted to showing that $\partial D$ has no cusps. In \textbf{Part 3.}, we construct a dense subset $U_N\subset\C^A$, where $\ovD$ has no nodes at $\partial D$ for any $f\in U_N$. The proof follows by taking $U_d:=U_1\cap U_2 \cap U_N$.

\textbf{Part 1.} From Lemmas~\ref{lem:Newton_polytopes-invar}, and~\ref{lem:C_Newton-non}, we have that $C$ irreducible, and thus so is $D$ irreducible. Since $C$ is the image of $\bcC:=\Vs(f_1 - w_1,~f_2 - w_2,~\det\Jac_z f )$ under the projection $\Proj:(\C^*)^4\to\TT$, $(z,w)\mapsto z$, the curve $D$ is the image of $\bcC$ under the second projection  $\pi:(\C^*)^4\to\TT$, $(z,w)\mapsto w$. Let $\overline{\bcC}$ denote the compactification of $\bcC$ in the toric variety $ X_\Sigma\times X_\Delta$, where $\Sigma:=\NP(\Csf)$ and $\Delta:=\NP(\Dsf)$. 

Theorem~\ref{thm:toric_product} shows that $X_\Sigma\times X_\Delta$ corresponds to the polytope $\Lambda:=\Sigma\times\Delta$. Then, using~\cite[Theorem 3.3.4]{cox2011toric}, we can extend $\Proj$ to the toric morphism $\overline{\Proj}:X_\Lambda\to X_\Sigma$, induced by the $\Z$-linear map $\psi:=\Proj_{|\Z^4}:\Z^4\to\Z^2$ compatible with the normal fans $\Gamma_{\Lambda}$ and $\Gamma_{\Sigma}$. Similarly, we get $\overline{\pi}:X_\Lambda\to X_\Delta$, induced by $\psi:=\pi_{|\Z^4}:\Z^4\to\Z^2$, compatible with $\Gamma_{\Lambda}$ and $\Gamma_{\Delta}$. We get $\Proj=\overline{\Proj}_{|(\C^*)^4}$ and $\pi=\ovpi_{|(\C^*)^4}$. From $\Proj(\bcC) = C$ and $\pi(\bcC) = D$, we get $\overline{\Proj}(\obcC) = \ovC$ and $\ovpi(\obcC) = \ovD$. Furthermore, from Theorem~\ref{thm:toric_morphisms}, we get $\overline{\Proj}(\partial\bcC ) = \partial C$ and $\ovpi(\partial\bcC )=\partial D$. This gives rise to the commutative diagram
\[
  \begin{tikzcd}
    (X_{\Lambda},~\obcC) \arrow{r}{\overline{\Proj}} \arrow[swap]{dr}{\ovpi} &     (X_{\Sigma},~\ovC) \arrow{d}{F} \\
     & (X_{\Delta},~\ovD),
  \end{tikzcd}
\] where $F: \ovC\to \ovD$ satisfies $f_{|C}=\overline{F}_{|C}$. 

Next, we describe the intersections of $\obcC$ with the orbits in $X_\Sigma\times X_\Delta\cong X_{\Lambda}$. Each such orbit can be expressed as $V(\sigma\times\delta)\cong V(\sigma)\times V(\delta)$, for some faces $\sigma\prec\Sigma$ and $\delta\prec\Delta$. On the one hand, since $\Sigma$ and $\Delta$ are the respective Newton polytopes of $C$ and $D$, the curves $\ovC\subset X_{\Sigma}$ and $\ovD\subset X_{\Delta}$ do not intersect the $0$-dimensional orbits in their respective canonical compactifications. Hence, orbits $V(\sigma\times\delta)$ corresponding to vertices, do not contain points of $\partial\bcC$. On the other hand, since both $\Proj(\bcC)$ and $\pi(\bcC)$ belong to the maximal orbit $\TT$ of $X_\Sigma$ and $X_\Delta$ respectively, orbits $V(\sigma\times\delta)$ of dimension higher than two also do not contain points of $\partial\bcC$. Therefore, Theorem~\ref{thm:toric_morphisms} yields
\begin{equation}\label{eq:discr-orbits}
\partial\bcC\in V(\sigma\times\delta)\quad\Longleftrightarrow\quad \partial C\in V(\sigma),~\partial D\in V(\delta)~\text{ and }~\dim \sigma = \dim\delta = 1.
\end{equation}
% Now, we use the above description to prove Theorem~\ref{thm:discr-infty}.
  Since $f\in U_2$, Remark~\ref{rem:singular-orbit-Newton-non} shows that $\ovC$ is smooth at $\partial{C}$, and thus $\obcC$ is smooth at $\partial \bcC$. Then, any $q\in\partial D$ is either smooth, is a node (i.e. $\#\ovpi_{|\obcC}^{-1}(q)\geq 2$), or a cusp (i.e. $\exists~\varrho\in\partial\cC$ smooth point, $q=\ovpi(\varrho)$, and $\varrho\in\Crit(\ovpi_{|\obcC})$). 

\textbf{Part 2.} We show that $\ovD$ has no cusps at $\partial D$: Assume by contradiction that $q$ is a cusp, and let $\bm{V}:=V(\sigma\times\delta)$ denote the $2$-orbit as in~\eqref{eq:discr-orbits} containing the critical point $\varrho=:(p,q)$. Theorems~\ref{thm:toric_product} and~\ref{thm:toric_morphisms} show that $\ovpi_{|\bm{V}}:\bm{V}\longrightarrow V(\delta)$ is also a projection forgetting $p$. Hence, the curve $\obcC$ intersects $\bm{V}$ non-transversally at $(p,q)$, and thus the tangent $T_\varrho\obcC$ is parallel to the orbit $V(\sigma)\subset X_{\Sigma}$. This shows that $\ovC$ does not intersect transversally the orbit $V(\sigma)$, which contradicts to Lemma~\ref{lem:C_Newton-non} according to Remark~\ref{rem:singular-orbit-Newton-non}. 

\textbf{Part 3.} We construct the dense subset $U_N\subset\C^A$ mentioned at the beginning. 
%Now, we find a dense subset $U_N\subset\C^A$ such that $\ovD$ has no nodes at $\partial D$ for any $f\in U_N$. The proof follows by taking $U_d:=U_1\cap U_2 \cap U_N$.
 For any $f\in\C^A$, a node $q\in\partial D$ of $\ovD$ satisfies $\ovpi(\varrho) = \ovpi(\varrho') = q$ for some distinct $\varrho,\varrho'\in \partial\bcC\cap X_{\Lambda}$. 
 Then, there are two curves $\varphi,\varphi':~]0,1[\longrightarrow C$ in two different branches of $C$, satisfying 
	\[
	\lim_{t\to 0} \overline{\Phi}_\Lambda\big(\varphi(t),f(\varphi(t))\big) =\varrho\quad\text{and}\quad\lim_{t\to 0} \overline{\Phi}_\Lambda\big(\varphi'(t),f(\varphi'(t))\big) = \varrho',
	\] where $\overline{\Phi}_\Lambda:=\varsigma\circ\Phi_\Lambda:(\C^*)^4\to X_\Lambda\subset\P^{\#\Lambda\cap \N^4}$ with notations from~\S\ref{sss:toric-lattice}. Since it holds $\overline{\Phi}_\Lambda(C)=\obcC$, the images of $F\circ \overline{\Phi}_\Sigma $ and of $\overline{\pi}\circ\overline{\Phi}_\Lambda$ coincide at the curves $\varphi(]0,1[)$ and $\varphi'(]0,1[)$. Thus, we obtain 
	\begin{align}
	\lim_{t\to 0} F\circ \overline{\Phi}_\Sigma(\varphi(t)) = \lim_{t\to 0} F\circ \overline{\Phi}_\Sigma(\varphi'(t)) &=  0\label{eqs:limits2-2}.	
	\end{align} Using the curve selection Lemma (see e.g.,~\cite{bochnak2013real} or~\cite[Lemma 3.1]{jelonek2005quantitative}) we replace the families $\varphi(t)$ and $\varphi'(t)$ by Puiseux series in $t$ with complex coefficients
	\begin{align*}
	\varphi(t) & = (a_1t^{\alpha_1} + b_1t^{\geq \alpha_1}+ \cdots,~a_2t^{\alpha_2} + b_2t^{\geq \alpha_1}+ \cdots ) \\
	\varphi'(t) & = (a'_1t^{\alpha'_1} + b'_1t^{\geq \alpha'_1}+ \cdots,~a'_2t^{\alpha'_2} + b_2t^{\geq \alpha'_1}+ \cdots ),
	\end{align*} where $\alpha:=(\alpha_1,\alpha_2)$ and $\alpha':=(\alpha'_1,\alpha'_2)$ are primitive integer vectors. Recall that each of $\varrho$ and $\varrho'$ belong to a $2$-orbit corresponding to proper faces of $\Sigma$ and $\Delta$. Then, we get $\alpha\neq (0,0)$ and $\alpha'\neq (0,0)$. Let $\sigma,\sigma'\prec\Sigma$ be the two faces supported by $\alpha$ and $\alpha'$ respectively. If we substitute $\varphi(t)$ in $Q$
and set to zero the  coefficient of the smallest power of $t$, say $\omega^*$,
then
we obtain $Q_{\sigma}(a_1,a_2)=0$.
Indeed, we have $\omega^*=\min(\langle\alpha,\omega\rangle~|~\omega\in \Sigma)$, attained only for points $w$ in $\sigma$. This also holds true for $\varphi'(t)$. 
Thus, we get  
	\begin{equation}\label{eq:solutions-to_Jac}
		a:=(a_1,~a_2)\in\Vs(Q_\sigma)\quad\text{and}\quad a':=(a_1',~a_2')\in\Vs(Q_{\sigma'}).
	\end{equation}
	
	Assume first that $\sigma\neq \sigma'$. Then, in a dense subset $U_{\sigma}\subset\C^A$, any choice of $f$ makes the coefficients in $Q_\sigma$ and $Q_{\sigma'}$ independent so that~\eqref{eqs:limits2-2} is not satisfied. Let $U_\Sigma$ denote the intersection of all such subsets $U_\sigma$. It follows that for maps $f\in U_\Sigma$, the canonical compactification $\ovD$ of $D$ cannot have nodes at $\partial D$ arising from two non-parallel branches. 
	
	 Assume in what follows that $\sigma= \sigma'$, and let $v\in\Z^2$ be a primitive integer vector directing the edge $\sigma$. Then, there exists $w\in\Z^2$ and $c_1,\ldots,c_k\in\C^*$, such that
\begin{equation*}
Q_\sigma = z^w\prod_{i=1}^k(z^v - c_i).
\end{equation*} Note that $c_1,\ldots,c_k$ are pairwise-distinct as $Q$ is Newton non-degenerate (Lemma~\ref{lem:C_Newton-non}). Furthermore, since $\varphi(t)$ and $\varphi'(t)$ are taken from two different branches of $C$, and since $Q$ is Newton non-degenerate, identities~\eqref{eq:solutions-to_Jac} imply that
\begin{equation}\label{eq:different_in_orbit}
a^v \neq (a')^{v}. 
\end{equation} 

Let $\gamma$ be the face of $A$ supported by $\alpha$ (recall that $\sigma=\sigma'$ $\Rightarrow$ $\alpha=\alpha'$). Then, we can express $f_{\gamma_1}$ and $f_{\gamma_2}$ as $z^rh_1(z^v)$ and $z^sh_2(z^v)$, where $h_1,h_2\in\C[x]$, and $r,s\in\N^2$. Using similar arguments as above for~\eqref{eq:solutions-to_Jac}, the (projective) coordinates of $q\in \partial C$ at the $1$-orbit $V(\delta)\cong \P^1$ are given by $f_\gamma(a)=(f_{\gamma_1}(a),~f_{\gamma_2}(a))$. Therefore, relation~\eqref{eqs:limits2-2} implies that $f_\gamma(a)$ and $f_\gamma(a')$ are the same point up to the torus action on $V(\delta)$. Taking the quotient by this action we get
\begin{equation}\label{eq:equal_in_orbit}
(f_{\gamma_1}(a))^{\lambda_1}\cdot (f_{\gamma_2}(a))^{\lambda_2} = (f_{\gamma_1}(a'))^{\lambda_1}\cdot (f_{\gamma_2}(a'))^{\lambda_2},
\end{equation} where $\lambda:=(\lambda_1,\lambda_2)\in\Z^2$ satisfies 
\begin{equation}\label{eq:diophantine}
\lambda_1\cdot r + \lambda_2\cdot s = m\cdot v
\end{equation} for some $m\in\N^*$. Indeed, for any $z\in\C^*$, the polynomial $(f_\gamma)^\lambda(z)$ is a univariate rational function $x^m\cdot h_1^{\lambda_1}(x)h_2^{\lambda_2}(x)$, where $x:=z^v$. The existence of such $\lambda\in\Z^2$ from~\eqref{eq:diophantine} is guaranteed by the Chinese Remainder Theorem.

We conclude from relations~\eqref{eq:solutions-to_Jac},~\eqref{eq:different_in_orbit} and~\eqref{eq:equal_in_orbit} that if $\ovC$ has a node in $\partial C$, then for some distinct $i,j\in [k]$, the system 
\begin{align*}
\zeta^v = & c_i\\
\theta^v = & c_j\\
(f_{\gamma}(\zeta))^{\lambda} - (f_{\gamma}(\theta))^{\lambda}=&0,
\end{align*} has a solution $(\zeta,\theta)\in (\C^*)^4$. Recall that $c_i$ and $c_j$ are taken from the expression of $Q_\sigma=\det\Jac_z f_\gamma$. Then, it is an easy computation to check that, for $x=z^v$, such a solution is equivalent to $H:=x^m\cdot h_1^{\lambda_1}(x)h_2^{\lambda_2}(x)$ having two critical points $x_0,x_1\in\C^*$ satisfying $H(x_0) = H(x_1)$. The set of polynomials $h_1,h_2\in\C[x]$ for which this condition is not satisfied forms a Zarisky open $O_\gamma$ inside $\C^{\gamma}$. 

To conclude, we take the above set $U_N$ to be the common locus of $U_\Sigma$ with the preimages of all $O_\gamma$ under the corresponding projections $\C^A\longrightarrow\C^\gamma$, $f\longmapsto f_{\gamma}$. This yields the proof.
\end{proof}

\begin{remark}\label{rem:rays}
In the notations of the proof of Theorem~\ref{thm:discr-infty}, we can show that the vector $\lambda$ directs a face of $\Delta$. Indeed, each branch of $D\subset \TT$, parametrized by $\psi:~]0,1[\to D$ is the image of a branch of $C$ under the map $f$. The leading terms of $f(\varphi(t))$ are written as
\[
\big(a^rt^{r_1\alpha_1+r_2\alpha_2}\cdot h_1(a^vt^{v_1\alpha_1+v_2\alpha_2}),~a^st^{s_1\alpha_1+s_2\alpha_2}h_2(a^vt^{v_1\alpha_1+v_2\alpha_2})\big).
\] Since $\alpha$ is orthogonal to $v$, there exists $d_1,d_2\in\C^*$ such that 
\[
f(\varphi(t))\sim(d_1 t^{\beta_1},~d_2 t^{\beta_2})\text{ as }t\rightarrow 0,
\] where 
\[
\beta_1 = r_1\alpha_1+r_2\alpha_2,~\beta_2 = s_1\alpha_1+s_2\alpha_2,~d_1=a^rh_1(a^v),\text{ and }d_2=a^sh_2(a^v).
\] Similarly as above for~\eqref{eq:solutions-to_Jac}, as $t\rightarrow 0$, the point $(d_1,d_2)$ is in the zero locus $\V_*(D_{\tilde{\delta}})$ for some edge $\tilde{\delta}\prec \Delta$, supported by $\beta:=(\beta_1,\beta_2)$. In fact, $q \in V(\delta)$ implies that $\delta=\tilde{\delta}$. Since $\alpha$ is orthogonal to $v$, it is an easy computation to check that $\beta$ is orthogonal to $\lambda$. It follows that $\lambda$ is a vector directing the edge $\delta\prec\Delta$. Notice that every edge of $\Delta$ can be characterized in this way.
\end{remark}

\section{Singular points of the non-properness set}\label{sec:non-prop}

In this section, we present the main result addressing Theorem~\ref{thm:discriminant}~\ref{it:non-properness}. For any pair $\Pi$ of polytopes $\Pi_1,\Pi_2\subset\R^2$ of polytopes, if $\gamma$ is a pair of faces $(\gamma_1, \gamma_2)$ such that $\gamma_1 \prec \Pi_1$, $\gamma_2 \prec \Pi_2$, and $\gamma_1 \oplus \gamma_2\prec \Pi_1 \oplus \Pi_2$, then we say that $\gamma$ is a \emph{face} of $\Pi$, and we write $\gamma\prec \Pi$. The following definition distinguishes these types of faces, while Figure~\ref{fig:pair-example} provides relevant illustrations.

\begin{definition}
  \label{def:various-faces}
  Let $A\in\ccPgeq$, let $\gamma:=(\gamma_1, \gamma_2)$ be a face of $A^0$ (recall notations before Theorem~\ref{thm:discriminant}), and let $\alpha:=(\alpha_1,\alpha_2)\in\Z^2$ be one of the supporting inner vectors of $\gamma_1\oplus\gamma_2\prec A_1^0\oplus A_2^0$. We say that $\gamma$ is 
  an \emph{edge} if $\dim(\gamma_1\oplus\gamma_2) = 1$. An edge is
  \begin{itemize} %[labelindent=5mm,topsep=0mm,noitemsep]
  \item \emph{long} if  $\dim(\gamma_1) = \dim(\gamma_2) = 1$,
  \item \emph{short} if it is not long, 

  \item \emph{semi-origin} if $\bm{0}:=(0, 0) \in H$, where $H$ is the supporting line of either $\gamma_1$ or of $\gamma_2$, and 
\begin{equation}\label{eq:condition_semi}
  \dim\gamma_i =0\Longrightarrow \bm{0}\in \gamma_i,~i\in\{1,2\},
\end{equation}
  \item \emph{origin} if it is semi-origin, and the above line $H$ supports both $\gamma_1$ and $\gamma_2$,
  \item \emph{dicritical} if it is semi-origin, and either $\alpha_1$ or $\alpha_2$ is negative,
  
    \item \emph{lower}/\emph{upper} if it is dicritical and $\alpha_1<0$/~$\alpha_2<0$.
 
  \end{itemize}
\end{definition}

\begin{example}\label{ex:edges}
The pair $A^0$ in Figure~\ref{fig:pair-example} admits two semi-origin edges: One, denoted by $\delta$ is represented in homogeneous dashing, it is long, upper dicritical, and not origin. The other, $\gamma$, is represented by a ``dot-dash'' pattern, and is also long, lower dicritical, and origin. The pair $A^0$ has at least one edge that is not semi-origin. For example, the short one, $(\gamma'_1,\gamma'_2)$, where $\gamma'_1$ is a dotted segment, and $\gamma'_2$ is the vertex square labelled ``$\square$'' situated at $(5,5)$.
\end{example}

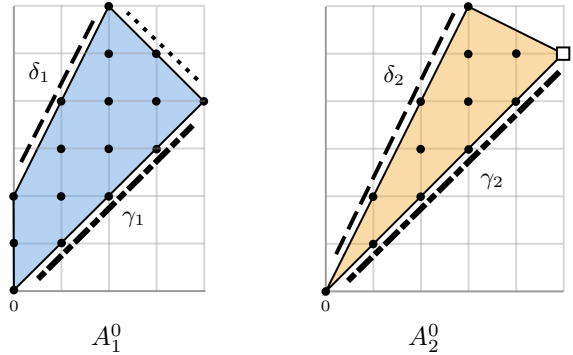
\begin{figure}[htb]

\tikzset{every picture/.style={line width=0.75pt}} %set default line width to 0.75pt        

\begin{tikzpicture}[x=0.75pt,y=0.75pt,yscale=-1,xscale=1]
%uncomment if require: \path (0,435); %set diagram left start at 0, and has height of 435

%Straight Lines [id:da04603395413484912] 
\draw [color={rgb, 255:red, 155; green, 155; blue, 155 }  ,draw opacity=0.4 ][fill={rgb, 255:red, 155; green, 155; blue, 155 }  ,fill opacity=0.3 ]   (212.33,156.15) -- (332.02,156.15) ;
%Straight Lines [id:da0775044335399303] 
\draw [color={rgb, 255:red, 155; green, 155; blue, 155 }  ,draw opacity=0.4 ][fill={rgb, 255:red, 155; green, 155; blue, 155 }  ,fill opacity=0.3 ]   (308.33,276.84) -- (308.33,132) ;
%Straight Lines [id:da9603564253636753] 
\draw [color={rgb, 255:red, 155; green, 155; blue, 155 }  ,draw opacity=0.4 ][fill={rgb, 255:red, 155; green, 155; blue, 155 }  ,fill opacity=0.3 ]   (212.33,204.15) -- (332.02,204.15) ;
%Straight Lines [id:da5616480709444268] 
\draw [color={rgb, 255:red, 155; green, 155; blue, 155 }  ,draw opacity=0.4 ][fill={rgb, 255:red, 155; green, 155; blue, 155 }  ,fill opacity=0.3 ]   (332.33,276.11) -- (332.33,131.27) ;
%Straight Lines [id:da878023135469378] 
\draw [color={rgb, 255:red, 155; green, 155; blue, 155 }  ,draw opacity=0.4 ][fill={rgb, 255:red, 155; green, 155; blue, 155 }  ,fill opacity=0.3 ]   (212.33,252.15) -- (332.02,252.15) ;
%Straight Lines [id:da8212509132033688] 
\draw [color={rgb, 255:red, 155; green, 155; blue, 155 }  ,draw opacity=0.4 ][fill={rgb, 255:red, 155; green, 155; blue, 155 }  ,fill opacity=0.3 ]   (236.33,275.65) -- (236.33,131.81) ;
%Straight Lines [id:da8226756057031955] 
\draw [color={rgb, 255:red, 155; green, 155; blue, 155 }  ,draw opacity=0.4 ][fill={rgb, 255:red, 155; green, 155; blue, 155 }  ,fill opacity=0.3 ]   (260.33,276.85) -- (260.33,132) ;
%Straight Lines [id:da9024323720410014] 
\draw [color={rgb, 255:red, 155; green, 155; blue, 155 }  ,draw opacity=0.4 ][fill={rgb, 255:red, 155; green, 155; blue, 155 }  ,fill opacity=0.3 ]   (284.33,276.49) -- (284.33,132.65) ;
%Straight Lines [id:da9897752801637711] 
\draw [color={rgb, 255:red, 155; green, 155; blue, 155 }  ,draw opacity=1 ]   (55,276) -- (151.57,276) ;
%Straight Lines [id:da6715179490283009] 
\draw [color={rgb, 255:red, 155; green, 155; blue, 155 }  ,draw opacity=0.4 ][fill={rgb, 255:red, 155; green, 155; blue, 155 }  ,fill opacity=0.3 ]   (55,252) -- (151.32,252) ;
%Straight Lines [id:da4498501655497258] 
\draw [color={rgb, 255:red, 155; green, 155; blue, 155 }  ,draw opacity=0.4 ][fill={rgb, 255:red, 155; green, 155; blue, 155 }  ,fill opacity=0.3 ]   (55,228) -- (151.32,228) ;
%Straight Lines [id:da006015979378478087] 
\draw [color={rgb, 255:red, 155; green, 155; blue, 155 }  ,draw opacity=1 ]   (55,276.05) -- (55,132.24) ;
%Straight Lines [id:da8418932439699712] 
\draw [color={rgb, 255:red, 155; green, 155; blue, 155 }  ,draw opacity=0.4 ][fill={rgb, 255:red, 155; green, 155; blue, 155 }  ,fill opacity=0.3 ]   (79,275.51) -- (79,131.67) ;
%Straight Lines [id:da5263085854069552] 
\draw [color={rgb, 255:red, 155; green, 155; blue, 155 }  ,draw opacity=0.4 ][fill={rgb, 255:red, 155; green, 155; blue, 155 }  ,fill opacity=0.3 ]   (103,276.7) -- (103,132.86) ;
%Straight Lines [id:da3610376068740687] 
\draw [color={rgb, 255:red, 155; green, 155; blue, 155 }  ,draw opacity=0.4 ][fill={rgb, 255:red, 155; green, 155; blue, 155 }  ,fill opacity=0.3 ]   (127,276.34) -- (127,132.5) ;
%Shape: Polygon [id:ds7371308872653604] 
\draw  [fill={rgb, 255:red, 74; green, 144; blue, 226 }  ,fill opacity=0.4 ][line width=0.75]  (102.83,132.22) -- (150.83,180.22) -- (55,275.51) -- (54.83,228.22) -- cycle ;
%Shape: Circle [id:dp3331062559344151] 
\draw  [fill={rgb, 255:red, 0; green, 0; blue, 0 }  ,fill opacity=1 ] (101.17,228.22) .. controls (101.17,227.3) and (101.91,226.55) .. (102.83,226.55) .. controls (103.75,226.55) and (104.5,227.3) .. (104.5,228.22) .. controls (104.5,229.14) and (103.75,229.88) .. (102.83,229.88) .. controls (101.91,229.88) and (101.17,229.14) .. (101.17,228.22) -- cycle ;
%Shape: Circle [id:dp7025589061877436] 
\draw  [fill={rgb, 255:red, 0; green, 0; blue, 0 }  ,fill opacity=1 ] (53.17,228.22) .. controls (53.17,227.3) and (53.91,226.55) .. (54.83,226.55) .. controls (55.75,226.55) and (56.5,227.3) .. (56.5,228.22) .. controls (56.5,229.14) and (55.75,229.88) .. (54.83,229.88) .. controls (53.91,229.88) and (53.17,229.14) .. (53.17,228.22) -- cycle ;
%Shape: Circle [id:dp4821027793532584] 
\draw  [fill={rgb, 255:red, 0; green, 0; blue, 0 }  ,fill opacity=1 ] (330.35,156.15) .. controls (330.35,155.23) and (331.09,154.48) .. (332.02,154.48) .. controls (332.94,154.48) and (333.68,155.23) .. (333.68,156.15) .. controls (333.68,157.07) and (332.94,157.81) .. (332.02,157.81) .. controls (331.09,157.81) and (330.35,157.07) .. (330.35,156.15) -- cycle ;
%Straight Lines [id:da5508770494279803] 
\draw [color={rgb, 255:red, 155; green, 155; blue, 155 }  ,draw opacity=0.4 ][fill={rgb, 255:red, 155; green, 155; blue, 155 }  ,fill opacity=0.3 ]   (151,276.5) -- (151,132.66) ;
%Shape: Circle [id:dp9307821567706168] 
\draw  [fill={rgb, 255:red, 0; green, 0; blue, 0 }  ,fill opacity=1 ] (149.17,180.22) .. controls (149.17,179.3) and (149.91,178.55) .. (150.83,178.55) .. controls (151.75,178.55) and (152.5,179.3) .. (152.5,180.22) .. controls (152.5,181.14) and (151.75,181.88) .. (150.83,181.88) .. controls (149.91,181.88) and (149.17,181.14) .. (149.17,180.22) -- cycle ;
%Shape: Circle [id:dp003686143503202066] 
\draw  [fill={rgb, 255:red, 0; green, 0; blue, 0 }  ,fill opacity=1 ] (125.33,204.37) .. controls (125.33,203.45) and (126.08,202.7) .. (127,202.7) .. controls (127.92,202.7) and (128.67,203.45) .. (128.67,204.37) .. controls (128.67,205.29) and (127.92,206.03) .. (127,206.03) .. controls (126.08,206.03) and (125.33,205.29) .. (125.33,204.37) -- cycle ;
%Straight Lines [id:da41455223056488566] 
\draw [color={rgb, 255:red, 155; green, 155; blue, 155 }  ,draw opacity=0.4 ][fill={rgb, 255:red, 155; green, 155; blue, 155 }  ,fill opacity=0.3 ]   (55,204) -- (151.32,204) ;
%Straight Lines [id:da09266098184565474] 
\draw [color={rgb, 255:red, 155; green, 155; blue, 155 }  ,draw opacity=0.4 ][fill={rgb, 255:red, 155; green, 155; blue, 155 }  ,fill opacity=0.3 ]   (55,180) -- (151.32,180) ;
%Straight Lines [id:da5909160233202989] 
\draw [color={rgb, 255:red, 155; green, 155; blue, 155 }  ,draw opacity=0.4 ][fill={rgb, 255:red, 155; green, 155; blue, 155 }  ,fill opacity=0.3 ]   (55,156) -- (151.32,156) ;
%Straight Lines [id:da9613801617396066] 
\draw [color={rgb, 255:red, 155; green, 155; blue, 155 }  ,draw opacity=0.4 ][fill={rgb, 255:red, 155; green, 155; blue, 155 }  ,fill opacity=0.3 ]   (55,132) -- (151.32,132) ;
%Shape: Circle [id:dp2527725671041774] 
\draw  [fill={rgb, 255:red, 0; green, 0; blue, 0 }  ,fill opacity=1 ] (77.17,228.22) .. controls (77.17,227.3) and (77.91,226.55) .. (78.83,226.55) .. controls (79.75,226.55) and (80.5,227.3) .. (80.5,228.22) .. controls (80.5,229.14) and (79.75,229.88) .. (78.83,229.88) .. controls (77.91,229.88) and (77.17,229.14) .. (77.17,228.22) -- cycle ;
%Shape: Circle [id:dp9861432855253378] 
\draw  [fill={rgb, 255:red, 0; green, 0; blue, 0 }  ,fill opacity=1 ] (101.17,132.22) .. controls (101.17,131.3) and (101.91,130.55) .. (102.83,130.55) .. controls (103.75,130.55) and (104.5,131.3) .. (104.5,132.22) .. controls (104.5,133.14) and (103.75,133.88) .. (102.83,133.88) .. controls (101.91,133.88) and (101.17,133.14) .. (101.17,132.22) -- cycle ;
%Shape: Circle [id:dp7769684144920551] 
\draw  [fill={rgb, 255:red, 0; green, 0; blue, 0 }  ,fill opacity=1 ] (77.17,180.22) .. controls (77.17,179.3) and (77.91,178.55) .. (78.83,178.55) .. controls (79.75,178.55) and (80.5,179.3) .. (80.5,180.22) .. controls (80.5,181.14) and (79.75,181.88) .. (78.83,181.88) .. controls (77.91,181.88) and (77.17,181.14) .. (77.17,180.22) -- cycle ;
%Shape: Circle [id:dp26047197843409087] 
\draw  [fill={rgb, 255:red, 0; green, 0; blue, 0 }  ,fill opacity=1 ] (101.17,156.22) .. controls (101.17,155.3) and (101.91,154.55) .. (102.83,154.55) .. controls (103.75,154.55) and (104.5,155.3) .. (104.5,156.22) .. controls (104.5,157.14) and (103.75,157.88) .. (102.83,157.88) .. controls (101.91,157.88) and (101.17,157.14) .. (101.17,156.22) -- cycle ;
%Shape: Circle [id:dp02054473417687308] 
\draw  [fill={rgb, 255:red, 0; green, 0; blue, 0 }  ,fill opacity=1 ] (125.17,156.22) .. controls (125.17,155.3) and (125.91,154.55) .. (126.83,154.55) .. controls (127.75,154.55) and (128.5,155.3) .. (128.5,156.22) .. controls (128.5,157.14) and (127.75,157.88) .. (126.83,157.88) .. controls (125.91,157.88) and (125.17,157.14) .. (125.17,156.22) -- cycle ;
%Shape: Circle [id:dp4647647698155045] 
\draw  [fill={rgb, 255:red, 0; green, 0; blue, 0 }  ,fill opacity=1 ] (125.17,180.22) .. controls (125.17,179.3) and (125.91,178.55) .. (126.83,178.55) .. controls (127.75,178.55) and (128.5,179.3) .. (128.5,180.22) .. controls (128.5,181.14) and (127.75,181.88) .. (126.83,181.88) .. controls (125.91,181.88) and (125.17,181.14) .. (125.17,180.22) -- cycle ;
%Shape: Circle [id:dp8484194265022958] 
\draw  [fill={rgb, 255:red, 0; green, 0; blue, 0 }  ,fill opacity=1 ] (101.17,180.22) .. controls (101.17,179.3) and (101.91,178.55) .. (102.83,178.55) .. controls (103.75,178.55) and (104.5,179.3) .. (104.5,180.22) .. controls (104.5,181.14) and (103.75,181.88) .. (102.83,181.88) .. controls (101.91,181.88) and (101.17,181.14) .. (101.17,180.22) -- cycle ;
%Shape: Circle [id:dp7373999511681396] 
\draw  [fill={rgb, 255:red, 0; green, 0; blue, 0 }  ,fill opacity=1 ] (101.17,204.22) .. controls (101.17,203.3) and (101.91,202.55) .. (102.83,202.55) .. controls (103.75,202.55) and (104.5,203.3) .. (104.5,204.22) .. controls (104.5,205.14) and (103.75,205.88) .. (102.83,205.88) .. controls (101.91,205.88) and (101.17,205.14) .. (101.17,204.22) -- cycle ;
%Shape: Circle [id:dp8883560834841918] 
\draw  [fill={rgb, 255:red, 0; green, 0; blue, 0 }  ,fill opacity=1 ] (77.17,204.22) .. controls (77.17,203.3) and (77.91,202.55) .. (78.83,202.55) .. controls (79.75,202.55) and (80.5,203.3) .. (80.5,204.22) .. controls (80.5,205.14) and (79.75,205.88) .. (78.83,205.88) .. controls (77.91,205.88) and (77.17,205.14) .. (77.17,204.22) -- cycle ;
%Straight Lines [id:da9151912003993694] 
\draw [color={rgb, 255:red, 155; green, 155; blue, 155 }  ,draw opacity=1 ]   (212.33,276.15) -- (332.32,276.15) ;
%Straight Lines [id:da047737571515244825] 
\draw [color={rgb, 255:red, 155; green, 155; blue, 155 }  ,draw opacity=0.4 ][fill={rgb, 255:red, 155; green, 155; blue, 155 }  ,fill opacity=0.3 ]   (212.33,228.15) -- (332.02,228.15) ;
%Straight Lines [id:da9130308450337771] 
\draw [color={rgb, 255:red, 155; green, 155; blue, 155 }  ,draw opacity=0.4 ][fill={rgb, 255:red, 155; green, 155; blue, 155 }  ,fill opacity=0.3 ]   (212.33,180.15) -- (331.88,180.15) ;
%Straight Lines [id:da5968989936069652] 
\draw [color={rgb, 255:red, 155; green, 155; blue, 155 }  ,draw opacity=1 ]   (212.33,276.2) -- (212.33,132.39) ;
%Shape: Polygon [id:ds03163545025864112] 
\draw  [fill={rgb, 255:red, 245; green, 166; blue, 35 }  ,fill opacity=0.4 ][line width=0.75]  (284.17,132.36) -- (332.02,156.15) -- (260.17,228.36) -- (212.33,276.2) -- cycle ;
%Shape: Circle [id:dp7625133091802337] 
\draw  [fill={rgb, 255:red, 0; green, 0; blue, 0 }  ,fill opacity=1 ] (258.5,228.36) .. controls (258.5,227.44) and (259.25,226.7) .. (260.17,226.7) .. controls (261.09,226.7) and (261.83,227.44) .. (261.83,228.36) .. controls (261.83,229.28) and (261.09,230.03) .. (260.17,230.03) .. controls (259.25,230.03) and (258.5,229.28) .. (258.5,228.36) -- cycle ;
%Shape: Circle [id:dp3060614434200183] 
\draw  [fill={rgb, 255:red, 0; green, 0; blue, 0 }  ,fill opacity=1 ] (306.5,180.36) .. controls (306.5,179.44) and (307.25,178.7) .. (308.17,178.7) .. controls (309.09,178.7) and (309.83,179.44) .. (309.83,180.36) .. controls (309.83,181.28) and (309.09,182.03) .. (308.17,182.03) .. controls (307.25,182.03) and (306.5,181.28) .. (306.5,180.36) -- cycle ;
%Shape: Circle [id:dp7998079288245934] 
\draw  [fill={rgb, 255:red, 0; green, 0; blue, 0 }  ,fill opacity=1 ] (282.67,204.51) .. controls (282.67,203.59) and (283.41,202.85) .. (284.33,202.85) .. controls (285.25,202.85) and (286,203.59) .. (286,204.51) .. controls (286,205.43) and (285.25,206.18) .. (284.33,206.18) .. controls (283.41,206.18) and (282.67,205.43) .. (282.67,204.51) -- cycle ;
%Straight Lines [id:da9167824690161729] 
\draw [color={rgb, 255:red, 155; green, 155; blue, 155 }  ,draw opacity=0.4 ][fill={rgb, 255:red, 155; green, 155; blue, 155 }  ,fill opacity=0.3 ]   (212.33,132.15) -- (332.02,132.15) ;
%Shape: Circle [id:dp7854484509760251] 
\draw  [fill={rgb, 255:red, 0; green, 0; blue, 0 }  ,fill opacity=1 ] (234.5,228.36) .. controls (234.5,227.44) and (235.25,226.7) .. (236.17,226.7) .. controls (237.09,226.7) and (237.83,227.44) .. (237.83,228.36) .. controls (237.83,229.28) and (237.09,230.03) .. (236.17,230.03) .. controls (235.25,230.03) and (234.5,229.28) .. (234.5,228.36) -- cycle ;
%Shape: Circle [id:dp10458619675374325] 
\draw  [fill={rgb, 255:red, 0; green, 0; blue, 0 }  ,fill opacity=1 ] (282.5,132.36) .. controls (282.5,131.44) and (283.25,130.7) .. (284.17,130.7) .. controls (285.09,130.7) and (285.83,131.44) .. (285.83,132.36) .. controls (285.83,133.28) and (285.09,134.03) .. (284.17,134.03) .. controls (283.25,134.03) and (282.5,133.28) .. (282.5,132.36) -- cycle ;
%Shape: Circle [id:dp4122829364171774] 
\draw  [fill={rgb, 255:red, 0; green, 0; blue, 0 }  ,fill opacity=1 ] (306.72,156.14) .. controls (306.72,155.22) and (307.47,154.47) .. (308.39,154.47) .. controls (309.31,154.47) and (310.06,155.22) .. (310.06,156.14) .. controls (310.06,157.06) and (309.31,157.81) .. (308.39,157.81) .. controls (307.47,157.81) and (306.72,157.06) .. (306.72,156.14) -- cycle ;
%Shape: Circle [id:dp3435449869376451] 
\draw  [fill={rgb, 255:red, 0; green, 0; blue, 0 }  ,fill opacity=1 ] (282.5,156.36) .. controls (282.5,155.44) and (283.25,154.7) .. (284.17,154.7) .. controls (285.09,154.7) and (285.83,155.44) .. (285.83,156.36) .. controls (285.83,157.28) and (285.09,158.03) .. (284.17,158.03) .. controls (283.25,158.03) and (282.5,157.28) .. (282.5,156.36) -- cycle ;
%Shape: Circle [id:dp9883726699748253] 
\draw  [fill={rgb, 255:red, 0; green, 0; blue, 0 }  ,fill opacity=1 ] (282.5,180.36) .. controls (282.5,179.44) and (283.25,178.7) .. (284.17,178.7) .. controls (285.09,178.7) and (285.83,179.44) .. (285.83,180.36) .. controls (285.83,181.28) and (285.09,182.03) .. (284.17,182.03) .. controls (283.25,182.03) and (282.5,181.28) .. (282.5,180.36) -- cycle ;
%Shape: Circle [id:dp3790770771800632] 
\draw  [fill={rgb, 255:red, 0; green, 0; blue, 0 }  ,fill opacity=1 ] (258.5,180.36) .. controls (258.5,179.44) and (259.25,178.7) .. (260.17,178.7) .. controls (261.09,178.7) and (261.83,179.44) .. (261.83,180.36) .. controls (261.83,181.28) and (261.09,182.03) .. (260.17,182.03) .. controls (259.25,182.03) and (258.5,181.28) .. (258.5,180.36) -- cycle ;
%Shape: Circle [id:dp25913469156369007] 
\draw  [fill={rgb, 255:red, 0; green, 0; blue, 0 }  ,fill opacity=1 ] (258.5,204.36) .. controls (258.5,203.44) and (259.25,202.7) .. (260.17,202.7) .. controls (261.09,202.7) and (261.83,203.44) .. (261.83,204.36) .. controls (261.83,205.28) and (261.09,206.03) .. (260.17,206.03) .. controls (259.25,206.03) and (258.5,205.28) .. (258.5,204.36) -- cycle ;
%Shape: Circle [id:dp4517147330359427] 
\draw  [fill={rgb, 255:red, 0; green, 0; blue, 0 }  ,fill opacity=1 ] (77.41,251.75) .. controls (77.41,250.83) and (78.16,250.09) .. (79.08,250.09) .. controls (80,250.09) and (80.75,250.83) .. (80.75,251.75) .. controls (80.75,252.67) and (80,253.42) .. (79.08,253.42) .. controls (78.16,253.42) and (77.41,252.67) .. (77.41,251.75) -- cycle ;
%Shape: Circle [id:dp12208787091367512] 
\draw  [fill={rgb, 255:red, 0; green, 0; blue, 0 }  ,fill opacity=1 ] (53.33,251.75) .. controls (53.33,250.83) and (54.08,250.09) .. (55,250.09) .. controls (55.92,250.09) and (56.67,250.83) .. (56.67,251.75) .. controls (56.67,252.67) and (55.92,253.42) .. (55,253.42) .. controls (54.08,253.42) and (53.33,252.67) .. (53.33,251.75) -- cycle ;
%Shape: Circle [id:dp10829195577269601] 
\draw  [fill={rgb, 255:red, 0; green, 0; blue, 0 }  ,fill opacity=1 ] (53.33,275.51) .. controls (53.33,274.59) and (54.08,273.84) .. (55,273.84) .. controls (55.92,273.84) and (56.67,274.59) .. (56.67,275.51) .. controls (56.67,276.43) and (55.92,277.17) .. (55,277.17) .. controls (54.08,277.17) and (53.33,276.43) .. (53.33,275.51) -- cycle ;
%Shape: Circle [id:dp7026743111144629] 
\draw  [fill={rgb, 255:red, 0; green, 0; blue, 0 }  ,fill opacity=1 ] (210.67,276.2) .. controls (210.67,275.28) and (211.41,274.53) .. (212.33,274.53) .. controls (213.25,274.53) and (214,275.28) .. (214,276.2) .. controls (214,277.12) and (213.25,277.87) .. (212.33,277.87) .. controls (211.41,277.87) and (210.67,277.12) .. (210.67,276.2) -- cycle ;
%Shape: Circle [id:dp35117402583361335] 
\draw  [fill={rgb, 255:red, 0; green, 0; blue, 0 }  ,fill opacity=1 ] (234.58,252.25) .. controls (234.58,251.33) and (235.33,250.59) .. (236.25,250.59) .. controls (237.17,250.59) and (237.92,251.33) .. (237.92,252.25) .. controls (237.92,253.17) and (237.17,253.92) .. (236.25,253.92) .. controls (235.33,253.92) and (234.58,253.17) .. (234.58,252.25) -- cycle ;
%Straight Lines [id:da5085561601196693] 
\draw [line width=2.25]  [dash pattern={on 3.75pt off 1.5pt on 7.5pt off 1.5pt}]  (67.17,269.88) -- (145.17,192.38) ;
%Straight Lines [id:da9042199426539458] 
\draw [line width=2.25]  [dash pattern={on 3.75pt off 1.5pt on 7.5pt off 1.5pt}]  (223.5,270.88) -- (330.62,164.85) ;
%Straight Lines [id:da632293632784402] 
\draw [line width=1.5]  [dash pattern={on 7.5pt off 3pt}]  (58.67,210.88) -- (95.17,139.38) ;
%Straight Lines [id:da3306112171065473] 
\draw [line width=1.5]  [dash pattern={on 7.5pt off 3pt}]  (217.5,257.38) -- (276.83,135.75) ;
%Straight Lines [id:da11673361040124153] 
\draw [line width=1.5]  [dash pattern={on 1.69pt off 2.76pt}]  (147.67,169.38) -- (112.17,135.38) ;
%Shape: Square [id:dp1480783306521556] 
\draw  [fill={rgb, 255:red, 255; green, 255; blue, 255 }  ,fill opacity=1 ] (328.88,153.01) -- (335.15,153.01) -- (335.15,159.28) -- (328.88,159.28) -- cycle ;

% Text Node
\draw (55,279.4) node [anchor=north] [inner sep=0.75pt]  [font=\tiny]  {$0$};
% Text Node
\draw (212.33,279.55) node [anchor=north] [inner sep=0.75pt]  [font=\tiny]  {$0$};
% Text Node
\draw (92.43,291.08) node [anchor=north west][inner sep=0.75pt]  [font=\normalsize]  {$A_{1}^{0}$};
% Text Node
\draw (252.43,291.01) node [anchor=north west][inner sep=0.75pt]  [font=\normalsize]  {$A_{2}^{0}$};
% Text Node
\draw (288.93,215.51) node [anchor=north west][inner sep=0.75pt]  [font=\small]  {$\gamma _{2}$};
% Text Node
\draw (108.17,234.53) node [anchor=north west][inner sep=0.75pt]  [font=\small]  {$\gamma _{1}$};
% Text Node
\draw (74.92,171.73) node [anchor=south east] [inner sep=0.75pt]  [font=\small]  {$\delta _{1}$};
% Text Node
\draw (253.92,173.23) node [anchor=south east] [inner sep=0.75pt]  [font=\small]  {$\delta _{2}$};

\end{tikzpicture}

\caption{Three different edges of $A^0$ (see Example~\ref{ex:edges}).}\label{fig:pair-example}
\end{figure} Let $\gamma:=(\gamma_1, \gamma_2) \prec A^0$, be a face. For any $f\in\C^A$, we
denote by $f_{\gamma_1}$ and $f_{\gamma_2}$ the restriction of $f_1$ and $f_2$ to those monomial
terms $c_{1, a} x^{a}$ for which  $a \in \gamma_1 \cap \Z^2$. That is,
\[
f_{\gamma_i} := \sum\nolimits_{a\in \supp f_i \cap \gamma_i} c_{i, a} ~z^{a},~i=1,2.
\] We also write $f_\gamma$ for the
pair of polynomials $f_{\gamma_1}, f_{\gamma_2}\in\C[z_1,z_2]$.

\begin{definition}\label{def:Newton_non-degenerate}
The pair of polynomials $f$ is said to be \emph{non-degenerate} at a face $\gamma\prec A^0$ if the restricted system $f_\gamma = 0$ has no solutions in $\TT$. We say that $f$ is \emph{Newton non-degenerate} if it is non-degenerate at every face of $\NP(f)$. 
%We also say that two algebraic curves in $\C^2$ are \emph{face-} or \emph{Newton-} \emph{non-degenerate} if so are all pairs of respective polynomials defining them.
\end{definition}
Now, consider the subset in $(\C^*)^4$ given by 
\[
G(\gamma,f):=\left\lbrace\left. (z,w)~\right|~(f-w)_{\gamma}(z)=0 \right\rbrace,
\] and the projection $\pi:(\C^*)^4\to\TT$, $(z,w)\mapsto w$. We use $\Sfg$ to denote the closure in $\C^2$ of $\pi(G(\gamma,f))$. The following result was shown in~\cite{EHT21}.

\begin{theorem}\label{thm:non-prop-main}
There is a dense subset $U_3\subset\C^A$ whose elements $f\in U_3$ represent polynomial maps $\CtC$ satisfying 
\[
\cS_f=\bigcup_{\gamma}\Sfg,
\] where $\gamma$ runs over all dicritical edges of $A^0$. 
\end{theorem}

\begin{example}\label{ex:non-prop}
Consider the polynomial map $f:\C^2_{x,y}\to\C_{a,b}^2$ from Example~\ref{ex:main}. The only two semi-origin faces of $A^0$ are $\gamma$ and $\delta$ appearing in Figure~\ref{fig:pair-example} (see e.g., Example~\ref{ex:edges}). We consider the pairs of polynomials
\begin{align*}
(f-w)_{\gamma}= &~\big(2~x^2y^2  -~3~x^4y^4 -~ a,~~ x^2y^2 -~11~x^5y^5 -~b\big) \\
(f-w)_{\delta}= &~\big( y^2~(-1 + x^2y^4),~~ - ~5~xy^2 +~x^2y^4 + ~7~x^3y^6-~b\big).
\end{align*} Then, we get $\cS_\delta(f)=\V\big((b-1)(b+11)\big)$. To compute $\cS_\gamma(f)$, we make the variable change $t:=xy$ to obtain its parametric equation $
\big(a(t),~b(t)\big) =(2t^2  -3t^4,~ t^2 -11t^5)$. Eliminating the variable $t$ gives the implicit equation for the polynomial $P$ from Example~\ref{ex:main}.
\end{example}

To prove the main result of this section, we need the following two Lemmas. The first one can be deduced by a straightforward computation. 

\begin{lemma}\label{lem:binom-non-degenerate}
If a polynomial $P$ has only two monomials, then $P$ is Newton non-degenerate.
\end{lemma} 

\begin{lemma}\label{lem:Newton_polytopes-invar2}
For any semi-origin edge $\gamma\prec A^0$, there exists a polytope $\SG\subset \R^2$, and a dense subset $U_4\subset\C^A$ in which each $f\in U_4$ satisfies $\SG = \NP(\Sfg)$.
\end{lemma}
\begin{proof}
Let $\gamma\prec A^0$ be a semi-origin edge. Theorem~\ref{thm:non-prop-main} shows that $\cS_f$ is a curve whose elimination ideal 
\[
\langle (f_1 - w_1)_{\gamma_1} ,~(f_2 - w_2)_{\gamma_2}\rangle\cap \Z[w_1,w_2]
\] is generated by a non-zero polynomial $\sum\phi_\theta ~w^{\theta}$ supported on a finite subset $\Theta$, and $\phi_\theta\in\Z[\underline{c},\underline{d}]$. Here, we take $\Gamma_\gamma:=\conv(\Theta)$ and $U_4:=\C^A\setminus\V(\prod\phi_\theta)$.
\end{proof}
The polytope $\Gamma_{\!\delta}$ from Example~\ref{ex:non-prop} is a segment, and $\Gamma_{\!\gamma}$ is illustrated in Figure~\ref{fig:JJJAD}.

\begin{theorem}\label{thm:non-properness-long}
Let $A\in\ccCgeq$. Then, there is a dense subset $U_s\subset\C^A$ containing all maps $f\in U_s$, such that for each long semi-origin edge $\gamma\prec A^0$, the curve $\Sfg$ is Newton non-degenerate whose singularities in $\C^2\setminus\{\bm{0}\}$ are simple nodes. Furthermore, for any two such edges $\gamma,\gamma'\prec A^0$, the pair $(\Sfg,\cS_{\gamma'}(f))$ is Newton non-degenerate, whose common locus in $\C^2\setminus\{\bm{0}\}$ is formed by complete intersections.
\end{theorem}

\begin{proof}
We take $U_s$ to be the intersection of subsets $U_\gamma$ consisting of maps satisfying the desired claims for each long semi-origin $\gamma\prec A^0$. Let $v\in\N^2$ be the primitive vector with positive integers directing $\gamma_1 \oplus \gamma_2$, for some semi-origin edge $\gamma:=(\gamma_1,\gamma_2)$.  Then, each of $f_{\gamma_1}(z)$ and $f_{\gamma_2}(z)$ can be expressed as univariate complex polynomials $h_1(x)$, and $h_2(x)$ respectively, where $x:=z^v$ (see e.g., Example~\ref{ex:non-prop}).

Assume first that $\gamma$ is not origin. Without loss of generality, we may suppose that $\bm{0}\in\gamma_1$. Then, $\Sfg$ is a union $\V(P)$ of vertical lines in $\C_{w_1,w_2}^2$, where
\begin{equation}\label{eq:vert-lines}
P:=\prod_{\substack{h_2(x)=0 \\ 
x\neq 0}} (w_1 -h_1(x)).
\end{equation} Consider the set $V_\gamma\subset\C^{\gamma}$ of all $f_\gamma$ such that $h_2$ does not have double roots, and $h_1(x)\neq h_1(y)$ for any distinct $x,y\in\V_*(h_2)$. Then, $V_\gamma$ is a Zariski open and $\Sfg$ is Newton non-degenerate. We finish this case by taking $U_\gamma$ to be the set of all $f$ for which $f_\gamma\in V_\gamma$. 

Assume now that $\gamma$ is an origin long edge. We thus get a parametrized expression 
\begin{equation}\label{eq:param-curve}
\Sfg=\overline{\left\{\left(h_1(x),~h_2(x)\right)~|~x\in\C^* \right\}}.
\end{equation} In what follows, we describe the nodes of $\Sfg$ using its parametrization~\eqref{eq:param-curve}.   

Let $\tilde{\gamma}$ denote the pair $(\gamma_1\cap A_1,~\gamma_2\cap A_2)$. First, note that the set of pairs $h:=(h_1,~h_2)\in\C^{\tilde{\gamma}}$ for which $\Sfg$ has non-nodal singularities (i.e. not self-intersections of smooth branches) forms a Zariski closed proper subset $Z$. Indeed, the latter is the image of the co-dimension two subset 
 \[
 \Vs(h'_1,~h'_2)\subset\C^{\tilde{\gamma}}\times\C^*
 \] under the projection $\C^{\tilde{\gamma}}\times\C^*\to\C^{\tilde{\gamma}}$. Therefore, if $V_s$ is defined to be the Zariski open
 \[
 \C^{A\setminus\tilde{\gamma}}\times\big(\C^{\tilde{\gamma}}\setminus Z\big), 
 \] then for any $f\in V_s$, all points in $\Sing(\Sfg \setminus\{\bm{0}\})$ are nodes. Next, we construct an open dense subset $W_s$ in which for any $f\in U_s:=W_s\cap V_s$, all the above nodes are simple. If $\nu\in \C^2\setminus \{\bm{0}\}$ is a node of $\Sfg$, there exists distinct $x,y\in\C^*$ satisfying 
\begin{alignat*}{3} % 4 is the number of equation columns
  \nu_1= & h_1(x) = & h_1(y),\\
  \nu_2= & h_2(x) = & h_2(y).
\end{alignat*} Hence, if $\nu$ is a simple node, it can be identified with a set of two distinct values $\{a,b\}$, where $(a,b)\in\TT$ (thus, also $(b,a)$) is a non-degenerate solution to 
\begin{equation}\label{eq:node}
\left\{ \begin{aligned} 
  h_1(x) - h_1(y) & = 0,\\
  h_2(x) - h_2(y) & = 0.
\end{aligned} \right.
\end{equation} The following claim shows that the singularities of $\Sfg$ in $\C^2\setminus\{\bm{0}\}$ are simple nodes.

Then, the first part of Theorem~\ref{thm:non-properness-long} follows from the following statement. 
\begin{claim}\label{clm:nodes-simple}
There is a Zariski open subset $W_s\subset \C^{\tilde{\gamma}}$ of pairs of polynomials $(h_1,h_2)\in W_s$ for which the solutions to~\eqref{eq:node} in $\TT\setminus\V(x-y)$ are non-degenerate. 
\end{claim}

\begin{proof}
Consider the projection 
\begin{align*}
\pi:\C^{\tilde{\gamma}}\times\TT&\longrightarrow \C^{\tilde{\gamma}}\\
(h_1,h_2,~x,y)& \longmapsto (h_1,h_2),
\end{align*} together with the set $Y:=\V\big(h_1(x) - h_1(y),~h_2(x) - h_2(y)\big)$. Then, from the definitions we get the identity
\[
\pi\big(\Crit \pi_{|_Y}\setminus\mathscr{D} \big)=\left\{(h_1,h_2)~|~\text{\eqref{eq:node} has a degenerate solution in $\TT\setminus\mathscr{D}$} \right\},
\] where $\mathscr{D}:=\V(x-y)$. Note that one can choose $(h_1,h_2)\in\C^{\tilde{\gamma}}$ so that for each $i=1,2$, if $x^k - y^k$ is a multiple of $h_i$, then $k=1$. Therefore, we get the identity
\[
\mathscr{D}=\left\lbrace (x,y)\in\TT~\text{satisfying~\eqref{eq:node} for all $(h_1,h_2)\in\C^{\tilde{\gamma}}$}\right\rbrace.
\] Finally, since $Y\setminus\mathscr{D}$ is smooth, Sard's Theorem implies that $
W_s:=\C^{\tilde{\gamma}}\setminus\pi\big(\Crit \pi_{|_{Y\setminus\mathscr{D}}} \big)$ is dense in $\C^{\tilde{\gamma}}$.
\end{proof} Now, we show that $\Sfg$ is Newton non-degenerate for all such $f\in U_s$. We first describe the set of edges of $\SG$. The implicit form of~\eqref{eq:param-curve} is given by a polynomial $P\in\C[w_1,w_2]$ generating the ideal
\[
\langle h_1(x)-w_1,~h_2(x) - w_2\rangle\cap\C[w_1,w_2].
\] Let $\{w^k\}_{k\in \N}\subset \Sfg\cap\TT$ be any family of points satisfying $ w^k\underset{k\rightarrow\infty}{\longrightarrow}\bm{0}$, or $\Vert w^k\Vert\underset{k\rightarrow\infty}{\longrightarrow}\infty$. By the curve-selection Lemma, we can approximate $\{w^k\}_{k\in \N}$ using the Puiseux series $\varphi:]0,1[\longrightarrow\TT$ with complex coefficients:
	\begin{align*}
	\varphi(t) & = (a_1t^{\alpha_1} + b_1t^{\beta_1\geq \alpha_1}+ \cdots,~a_2t^{\alpha_2} + b_2t^{\beta_2\geq \alpha_1}+ \cdots ).
	\end{align*} Let $(c_{m_1}x^{m_1},~d_{m_2}x^{m_2})$ and $(c_{n_1}x^{n_1},~d_{n_2}x^{n_2})$ be the pairs of terms in $(h_1,~h_2)$ of the lowest and highest degrees respectively. Since $\varphi(]0,1[)\subset \Sfg$, if $\varphi(t)\underset{t\rightarrow 0 }{\longrightarrow}\bm{0}$, we get 
	\[
	c_{m_1}(a_1t^{\alpha_1})^{m_1} = d_{m_2}(a_2t^{\alpha_2})^{m_2},\] for any small enough $t$. By the Newton-Puiseux theorem applied to $P$, we get a solution $( a_1,~a_2)$ to $P_{\sigma_+}=0$ for some $\sigma_+\prec\SG$. Since $(m_1,m_2)$ uniquely determines $(\alpha_1,\alpha_2)$ up to scaling, the edge $\sigma_+$ is the only one whose inner normal vector has positive coordinates. Analogously, if $\varphi_1(t)\underset{t\rightarrow 0 }{\longrightarrow}\infty$ or $\varphi_2(t)\underset{t\rightarrow 0 }{\longrightarrow}\infty$, the equality $c_{n_1}(a_1t^{\alpha_1})^{n_1} = d_{n_2}(a_2t^{\alpha_2})^{n_2}$ holds. Similarly, the point $( a_1,~a_2)$ is a solution to $Q_{\sigma_-}=0$, where $\sigma_-$ is the only edge of $\SG$ whose inner normal vector has a negative coordinate. We conclude that $\SG$ can have at most four edges $\sigma_-,\sigma_+,\sigma_{\rightarrow},\sigma_{\uparrow}\prec\SG$, where (In Figure~\ref{fig:JJJAD}, for example, we have $m_1=m_2 = 2$, $n_1=4$, and $n_2=5$)
\begin{equation}\label{eq:polytope-non-prop}
\begin{matrix}
\sigma_{\uparrow}:= & \conv\{(m_2,0),~(n_2,0)\}&  &\sigma_{\rightarrow}:= & \conv\{(0,m_1),~(0,n_1)\}\\ 
\sigma_-:= & \conv\{(n_2,0),~(0,n_1)\} & & \sigma_+:= & \conv\{(m_2,0),~(0,m_1)\}
\end{matrix}
\end{equation}	

Note that $P$ is non-degenerate at $\sigma_{\uparrow}$ and at $\sigma_{\rightarrow}$ since $h_1$ and $h_2$ are chosen so that each doesn't have double roots. As for the remaining two edges, the description above implies that $Q_{\sigma_-}$ and $Q_{\sigma_+}$ are written as $c_{n_1}(w_1)^{n_1} - d_{n_2}(w_2)^{n_2} $ and $c_{m_1}(w_1)^{m_1} - d_{m_2}(w_2)^{m_2}$. Then, the proof follows from Lemma~\ref{lem:binom-non-degenerate}.

Lastly, we show that for any long semi-origin faces $\gamma,\gamma'\prec A^0$, and any $f\in U_\gamma\cap U_{\gamma'}$, the curve $\Sfg$ intersects $\cS_{\gamma'}(f)$ transversally outsice $\{\bm{0}\}$, and $(\Sg,\cS_{\gamma'})$ is a Newton non-degenerate pair. The latter statement is straightforward since the dense subsets $U_\gamma$ and $U_{\gamma'}$ can be chosen so that $f_{\gamma}$ and $f_{\gamma'}$ do not share monomial terms. As for the former statement, using Bertini Theorem as in the proof of Claim~\ref{clm:nodes-simple}, the above choice for $U_\gamma$ and $U_{\gamma'}$ can be made to also guarantee transversality of $\Sg\cap\cS_{\gamma'}$ outside $\{\bm{0}\}$.
\end{proof}
\section{Proof of Theorem~\ref{thm:discriminant}}\label{sec:proof_of_theorem}
We start by presenting three technical results that take advantage of Newton non-degeneracy. These are particular cases of classical results of Kouchnirenko~\cite{Kou76}, Bernstein~\cite{Ber75}, and Khovanskii~\cite{khovanskii1978newton} respectively.
\begin{theorem}[Kouchnirenko]\label{thm:Kouchnirenko}
For any convenient bivariate complex polynomial $P$ it holds $\mu_0(P)\geq \cN(P)$, and equality holds if $P$ is Newton non-degenerate.
\end{theorem} 

\begin{theorem}[Bernstein]\label{thm:BKK}
Let $g_1,g_2\in\C[z_1,z_2]$ be two polynomials such that $(g_1,g_2)$ forms a Newton non-degenerate pair of polynomials. 
Then, the number of isolated solutions, counted with multiplicities, to the system $g_1=g_2=0$ outside $\{\bm{0}\}$ is equal to the mixed volume $\MV(\NP(f_1),\NP(f_2))$.
\end{theorem}

\begin{theorem}[Khovanskii]\label{thm:Khovanskii_genus2}
Let $D$ be a curve in $\TT$ with genus $\iota_D$ and Newton polytope $\Delta$, assume that $D$ has at worst mild singularities and that any $q\in \partial D$ is a smooth point of $\ovD\subset X_\Delta$ in its canonical toric compactification (see~\S\ref{sub:discr-orbits}). Then, we get
\begin{equation}\label{eq:genus-sing}
\#\Sin(D) = \hcir\Delta - \iota_{D}.
\end{equation}

\end{theorem}

\begin{proof}
We follow similar steps as in~\cite[Proof of Theorem 2.36]{esterov2013discriminant}. Assume that $P\in\C[z_1,z_2]$ satisfying $D=\Vs(P)$. Then, for any $a$ in $\Delta\cap\Z^2$, and for a small enough $\varepsilon\in\C$, the polynomial $P +\varepsilon z^a$ defines a smooth curve $S$ in $\TT$, close to $D$. Since $\ovD$ was smooth on $\partial D$, so will be $\overline{S}$. Then, the difference between the Euler characteristics of $\ovS$ and $\ovD$ is equal to the sum of the Milnor numbers of the singularities of $\ovC$ (see e.g.~\cite{khovanskii1978newton})
\begin{equation}\label{eq:Euler-singularities}
\chi(\ovD)- \chi(\ovS)   = \sum_{p\in\ovD} \mu_p(\ovD).
\end{equation} Let $\kappa$ and $\nu$ denote the number of cusps and nodes respectively of $D$. Since $\ovD$ is smooth at $\partial D$ and $\#\Sin D=\kappa+\nu$, the equality $\chi(\ovD) = 2 - 2\cdot \iota_{\ovD} - \nu$ holds and the right-hand side of~\eqref{eq:Euler-singularities} equals $2\cdot\kappa + \nu$. Furthermore, Applying the classical formulas in~\cite{khovanskii1978newton} to the smooth curve $\ovS$ yields
\[
\chi(\ovS) = 2 - 2\cdot \iota_{S} = 2 - 2\cdot \hcir\Delta.
\] Equation~\eqref{eq:genus-sing} is a follows by combining the above equalities $
2 - 2\cdot \iota_{D} - \nu -(2 - 2\cdot \hcir\Delta) = \nu+2\cdot\kappa$.
\end{proof}

Now, we are ready to prove Theorem~\ref{thm:discriminant}. Let $A\in\ccCgeq$, and let $U_n\subset\C^A$ denote the set of Newton-non-degenerate pairs $f\in\C^A$. It is known (see e.g.,~\cite{Kho16}) that $U_n$ is dense in $\C^A$. We set $\Omega_A$ to be the dense set in $\C^A$ given by
\[
\Omega_A:=U_c\cap U_s\cap U_d\cap U_n
\] involving subsets appearing in Theorems~\ref{thm:critical-smooth},~\ref{thm:discr-infty}, and~\ref{thm:non-properness-long}. Let $f\in\C^A$ be any element in $\Omega_A$.

\begin{proof}[Proof of Item~\ref{it:top-multiplicity}] 
Let $K_f$ denote the subset of all $w\in\C^2$ in which the set
\[
\V_*\big((f-w)_\gamma,~\partial (f-w)_\gamma/\partial z_1,~\partial(f-w)_\gamma/\partial z_2\big).
\] is empty for each semi-origin face $\gamma\prec A^0$. It can be checked that $K_f$ is Zariski open in $\C^2$. Then, the subset $K_f\times\Omega$ is dense in $\C^{A^0}$ and consists of Newton non-degenerate pairs of polynomials. Theorem~\ref{thm:BKK} implies that for every $w\in K_f$, the set $\V(f-w)$ has $\MV(A^0)$ isolated points in $\C^2\setminus\{\bm{0}\}$. If $K_f$ does not include the coordinate axes, then the inclusion $\V(f-w)\subset\C^2\setminus\{\bm{0}\}$ holds, and thus $\#\V(f-w)$ is the topological degree.
\end{proof}

\begin{proof}[Proof of Item~\ref{it:crit_genus}] All but the last assertion follow from Theorem~\ref{thm:critical-smooth}~\ref{it:smoothness}, Lemmas~\ref{lem:Newton_polytopes-invar},~\ref{lem:C_Newton-non}, and Theorem~\ref{thm:Kouchnirenko}. The genus of a smooth, Newton non-degenerate, curve is computed using Theorem~\ref{thm:Khovanskii_genus2} thanks to Remark~\ref{rem:singular-orbit-Newton-non}.\end{proof}

\begin{proof}[Proof of Item~\ref{it:crit_axes}]
This follows from simple computations and Remark~\ref{rem:gaps}.
\end{proof}

\begin{proof}[Proof of Item~\ref{it:discr_cusps}]
The first assertion is Lemma~\ref{lem:Newton_polytopes-invar}. Irreducibility and topological multiplicity of $\Dsf$ follow from Theorem~\ref{thm:discriminant}~\ref{it:crit_genus}. The genus of $\Dsf$ is deduced from Theorem~\ref{thm:critical-smooth}~\ref{it:genus=genus} and Theorem~\ref{thm:discriminant}~\ref{it:crit_genus}. Theorem~\ref{thm:critical-smooth}~\ref{it:cusps-nodes} asserts that all singularities of $\Dsf$ outside $\{\bm{0}\}$ are mild, and their number is computed using Theorem~\ref{thm:Khovanskii_genus2} thanks to Theorem~\ref{thm:discr-infty}.
\end{proof}

\begin{proof}[Proof of Item~\ref{it:non-properness}]
Let $\mathscr{G}(A)$ be the set of all semi-origin dicritical edges of $A^0$. It follows from the definitions that $\Sg:=\Sfg$ lies in $\C^2\setminus\TT$ if and only if $\gamma\in \mathscr{G}(A)$ is short. Furthermore, there are at most two long edges in $\mathscr{G}(A)$ (see e.g., Figure~\ref{fig:pair-example}). Then, Theorem~\ref{thm:non-prop-main} shows that
\[\cS_f^* = 
 \begin{cases}
  \Sg\cup \cS_{\gamma'}, & \text{if $\mathscr{G}(A)$ has exactly two long edges $\gamma$ and $\gamma'$}\\
  \Sg, & \text{if $\mathscr{G}(A)$ has exactly one long edges $\gamma$}\\
  \emptyset, & \text{otherwise}.
\end{cases} 
\] In what follows, thanks to Lemma~\ref{lem:Newton_polytopes-invar2}, we set $\cR  := \cS_\gamma$, $\cR'  := \cS_{\gamma'}$, $\Gamma :=\Gamma_\gamma$, and $ \Gamma':=\Gamma_{\gamma'}$. Then, the result follows from Theorems~\ref{thm:non-properness-long},~\ref{thm:BKK}, and~\ref{thm:Khovanskii_genus2}.
\end{proof}

\section{On computing the polyhedral type}\label{sec:algorithms}
In this section, we present lower bounds on the number of polyhedral types of conical pairs $A:=(A_1,A_2)\in\ccCgeq$ contained in $k$ times the affine $2-$simplex: $k\cdot~\sigma =  \{ (x_1, x_2) \in \R^2_{\geq 0} \mid x_1+x_2 \leq k\}$.
We achieve these lower bounds with the aid of software.
Consequently, we obtain the proof of Theorem~\ref{thm:main-bound} thanks to Theorem~\ref{thm:main2}. Our software and data 
are made available in the {\tt MathRepo} collection at MPI-MiS via
 \url{https://mathrepo.mis.mpg.de/PolyhedralTypesOfPlanarMaps}.
The source code can be downloaded at \url{https://github.com/kemalrose/PolyhedralTypesOfPlanarMaps}.
The following functionality is provided:
\begin{itemize}
    \item Listing all planar polytopes of a bounded degree.
    \item Computing the polytope $\Delta$ corresponding to a conical pair $A\in\ccCgeq$ (see Theorem~\ref{thm:discriminant}).
%    \item Computing the polyhedral type $\Psi(A)$ (see Definition~\ref{def:polyhrdral-type}).
\end{itemize}
After giving some examples, the rest of this section is devoted to explaining the algorithms for the above bullet points. These are the only non-trivial components for computing the topological type $\Psi(A)$. Indeed, the values of $\Psi(A)$ regarding mixed volumes, and Newton numbers have a straightforward implementation using the definitions, and equations for polytopes $\Gamma$ and $\Gamma'$ in Theorem~\ref{thm:discriminant}~\ref{it:non-properness} can be found at the end of the proof of Theorem~\ref{thm:non-properness-long}.

\begin{example}
Our software can be downloaded and activated using the following \emph{Julia} commands:
\begin{verbatim}
    using Pkg;
    Pkg.add(url = "https://github.com/kemalrose/PolyhedralTypes.jl");
    using PolyhedralTypes;
\end{verbatim}
The following line of code lists the planar polytopes contained in $k\cdot~\sigma$ for $k = 1, \dots, 9$.
There are
$(1, 27, 232, 1473, 8273, 43385, 214946, 1013317, 4559702)$ polytopes respectively.
\begin{verbatim}
    pol_Lists = list_all_polygons.([1,2,3,4,5,6,7,8,9])
\end{verbatim}
Filtering those that contribute to conical pairs, we confirm their respective number to be $(0,1,68,899,6795)$:
\begin{verbatim}
    conic_pols = [filter(verts->is_conic(verts), pol_list) for pol_list in pol_Lists]
\end{verbatim}
We obtain $0,1,2346,404550,23089410$ for the total numbers of conic pairs $A\in\ccCgeq$ with respective degree $\leq 1,2,3,4,5$, and up to interchanging $A_1$ and $A_2$.
\end{example}

\begin{example}
The following \texttt{Julia} code demonstrates our software. It recovers both the polyhedral type and the Newton polytope $\Delta$ of the discriminant from Example \ref{ex:main}.
\begin{verbatim}
    A1 = [0 2 4 2; 2 2 4 6];
    A2 = [1 2 5 3; 2 2 5 6];
    Delta = get_delta(A1, A2);
    println(Delta.pm_polytope.POINTS);
    psi = get_polyhedral_type(A1,A2);
    println(psi);
\end{verbatim}
\end{example}

  We remark that computing $A$-discriminants is a hard problem. Both their degrees and the size of their coefficients grow exceedingly large. Even examples of small size can be untractable for Gr\"obner basis and eliminination methods.
    Furthermore, resorting to computation over finite fields is unreliable when computing the Newton polytope.
    We encountered examples were, even for large finite fields $\mathbb{F}_p$,
    the Newton polytope of the rational Discriminant could not be recovered using Gr\"obner basis methods, since some of the rational coefficients where divisible by $p$.

\subsection{Listing planar polytopes of bounded degree}\label{sub:listing}
\label{subsection: listing planar polytopes}
The following pseudocode is based on \cite{BOUSQUETMELOU19961}. Given a finite subset $S$ of the lattice $\Z^2$, it lists all polygons whose vertices are contained in $S$.
A polygon is represented uniquely by its vertices, ordered clockwise and starting with the lexicographically minimal element.
In particular, every polygon is identified with a word in $S$.
We denote by $S^*$ the set of words with letters in $S$,
and further we denote by $ \mathcal{P} \subseteq  S^*$ the set of words representing a polygon.
We denote by
$\star : S^* \times S^* \longrightarrow S^*, \ (\omega_1, \dots, \omega_k)\star(\kappa_1, \dots, \kappa_l) = (\omega_1, \dots, \omega_k, \kappa_1, \dots, \kappa_l)$
the concatenation operator.
This induces a partial order on $S^*$ and $\mathcal{P}$:
for words $\omega, \kappa$ we set $\omega \leq \kappa$ if
$\omega$ is a prefix of $\kappa$:
\[
\omega \leq \kappa := \exists h \in S^* \mid \ \kappa = \omega \star h.
\]

Given lattice points $w_1$ and $w_2$ we denote by $H_{w_1,w_2}$ the open affine half space that is to the right of the line from $w_1$ to $w_2$.
Let $\omega =  w_1 w_2 \cdots w_{n-1}w_n $ in $\mathcal{P}$ represent a polygon with at least three vertices and let $v$ be an element of $S$. We say that $v$ is admissible for $\omega$ if:
\[
v \in H_{w_1,w_2} \cap H_{w_1,w_n} \cap H_{w_{n-1},w_n}, \ w_1 <_{ \operatorname{lex} } v.
\]

\begin{figure}[h]
\centering
\includegraphics[width = 6cm]{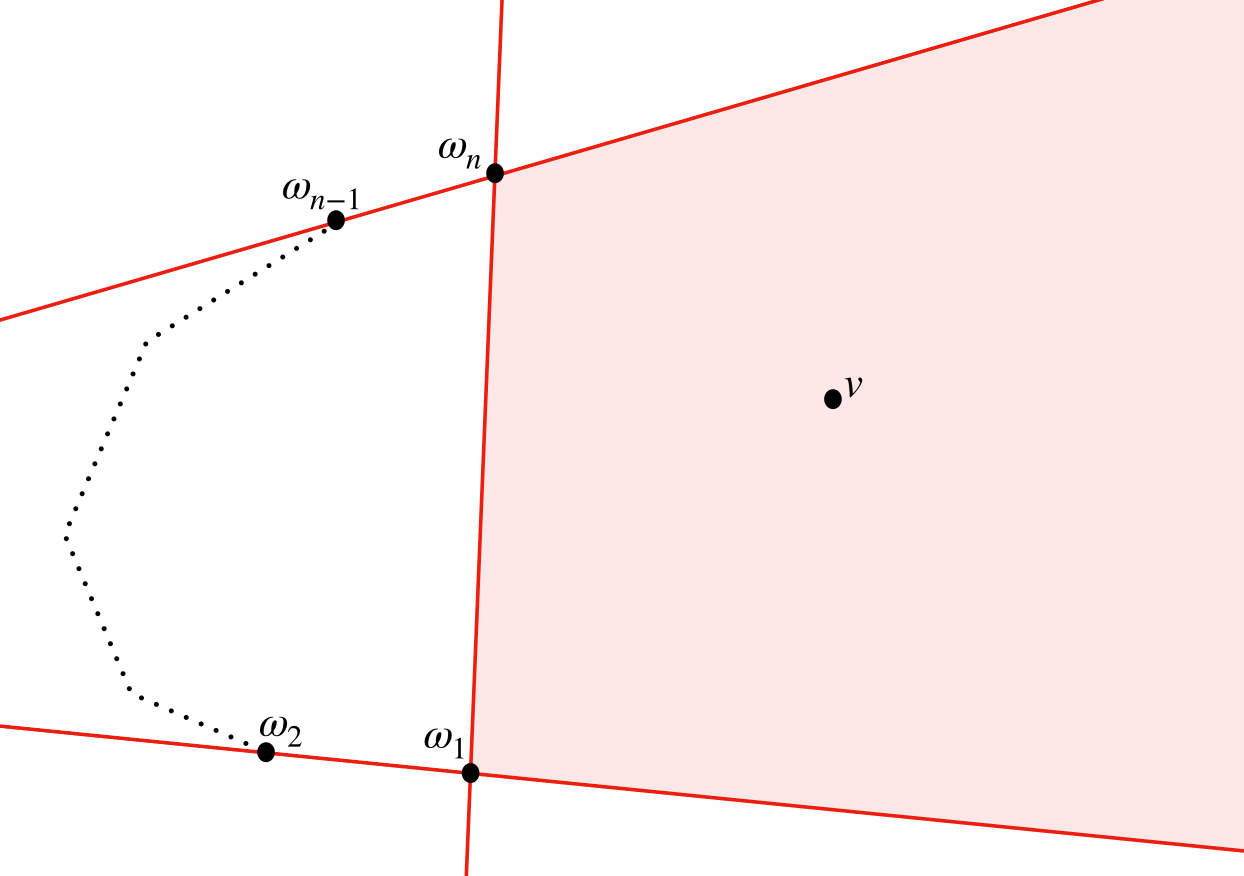} \vspace{-0.14in}
\caption{The halfspaces $H_{w_1,w_2}$, $H_{w_1,w_n}$ and $H_{w_{n-1},w_n}$ intersect in the red area.}
\label{fig: classical vs tropical implicitization}
\end{figure}

Our algorithm is based on the following observation made in \cite{BOUSQUETMELOU19961}:
let $\omega \in \mathcal{P}$ with at least three vertices, and $s\in S$.
\begin{proposition}
    The word $\omega \star(s)$ is an element of $\mathcal{P}$ 
    if and only if $s$ is admissible for $\omega$.
\end{proposition}

Our algorithm computes all elements of $\mathcal{P}$ by performing a
breadth-irst search through the underlying graph of the ordering $\leq$.
\begin{algorithm}
\caption{Listing bounded polygons}\label{euclid}
\begin{algorithmic}[1]
\Procedure{listPolytopes(S)}{}
\State $    \textit{polytopeList} \gets \emptyset$
\State $   \textit{newPolytopes}  \gets \{ (s_1, s_2) \ \operatorname{for} \ s_1 \in S, \ s_2 \in S,\ s_1 < s_2\}$
\While {$ \#(\textit{newPolytopes})$ > 0} 
\State        $\textit{polytopeList} \gets \textit{polytopeList} + \textit{newPolytopes}$
\State $\textit{extendedPolytopes} \gets \emptyset$
\For{ $\textit{polytope} \in \textit{newPolytopes}$}
\For{ $s \in S$}
\If  { $\operatorname{isAdmissible}(polytope, s)$}
\State  $\textit{extendedPolytopes}\gets \textit{extendedPolytopes} \cup \{ polytope+s \}$
\EndIf
\EndFor
\EndFor
\State $\textit{newPolytopes} \gets \textit{extendedPolytopes}$
\EndWhile
\State \Return \textit{polytopeList}
\EndProcedure
\end{algorithmic}
\end{algorithm}

%This algorithm is an effective implementation of the results in \ref{Boulos'smart work}.
%Given two polygons $P_1$ and $P_2$ in the plane it computes the rays of the inner normal fan of
%the Newton polytope $\Delta$ of the discriminant.
%It is based on solving the implicitization problem
%of computing the image $F(C)$ of the critical locus of $F$
%tropically.
%In general, tropicalization does not commute with taking images:
%The tropicalization of the image $\mathbb{T} (F(C))$ can be a strict subset of the
%image of the tropicalization $\mathbb{T}(C) $ under the tropical polynomial map $\mathbb{T} F = (\mathbb{T} f1, %\mathbb{T} f2)$.
%We quantified this difference in Proposition ...:
%The normal fan of $\Delta(P1, P2)$ is the union of the images of the rays
%in the normal fan of the Jacobian $\Jac(P1, P2)$ under the piecewise linear map $\mathbb{T} F $ and the rays spanned by %the negative vectors of unity $-e_1$ and $-e_2$.
%This leads to a simple algorithm to compute the rays of the normal fan of $\Delta(P1, P2)$.

\subsection{The Newton polytope of the discriminant} Here we present an algorithm for computing the polytope $\Delta$ corresponding to a $A\in\ccCgeq$ (see Theorem~\ref{thm:discriminant}). Its correctness requires some preliminary results.

\subsubsection{The lifted polytope}\label{sss:prel-correct}

 Recall that $A^0$ can have exactly two, one, or no long dicritical edges that are not origin. 
In the first case, we denote these edges by $\gamma:=(\gamma_1,\gamma_2)$ and $\delta:=(\delta_1,\delta_2)$, and in the second case, we denote it by $\gamma\prec A$ (see e.g., Figure~\ref{fig:semi-origins}). We fix once and for all the indexes $i,j\in\{1,2\}$ that satisfy $\bm{0}\not\in\gamma_i$ and $\bm{0}\not\in\delta_j$. 

\begin{figure}[htb]
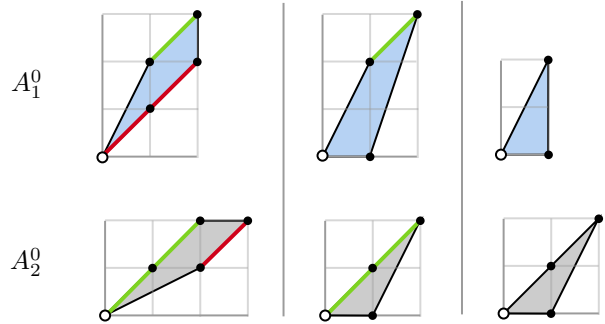


\tikzset{every picture/.style={line width=0.75pt}} %set default line width to 0.75pt        



\caption{Three examples of pairs $A^0$: The left-most has two dicritical long semi-origin edges that are not origin, the middle pair has one and the right-most has none.}\label{fig:semi-origins}
\end{figure}

\begin{definition}\label{def:correction_term}
Let $\Pi$ be a \emph{rational polytope} in $\R_{\geq}^2$. That is, the set of vertices $\Ver(\Pi)$ of $\Pi$ forms a subset of $\Q^2$. Its \emph{lifted polytope under $A$}, denoted by $\Lft(\Pi)$, is defined as the convex hull $\conv(K)$ of the set
\[
K:= \bigcup_{b:=(b_1,b_2)\in \Ver(\Pi)} \left(b_1\cdot A^0_1~\oplus~b_2~\cdot A^0_2 \right), 
\] where for any $\lambda\in\R$, $S\subset\R^n$, we write $
\lambda~\cdot~S:=\{\lambda\cdot s~|~s\in S\}$. 
 The \emph{$A$-correction term of $\Pi$} is defined as
\[
\cT_A(\Pi):= \max_{b\in\Ver(\Pi)}\{b_i~(\#~\gamma_i\cap\N^2 - 1)~ +~ b_j~(\#~\delta_j\cap\N^2 - 1)\},
\] where we set $\#~\gamma_i\cap\N^2 - 1:=0$ or  $\#~\delta_i\cap\N^2 - 1:=0$ if $A^0$ has less than two long dicritical edges that are not semi-origin.
\end{definition} 

\begin{theorem}\label{thm:correction_term-discriminant}
Let $A\in\ccCgeq$ be a dicritical pair, and let $\Sigma$ and $\Delta$ be the two polytopes, corresponding to $A$, appearing in Theorem~\ref{thm:discriminant}. For any lattice polytope $\Pi\subset\R_{\geq 0}^2$, the following equality holds
\begin{equation}\label{eq:correction_term-discriminant}
\MV(\Delta,~\Pi) = \MV(\Sigma,~\Lft(\Pi)) - \cT_A(\Pi).
\end{equation}
\end{theorem}

\begin{proof} 
We follow similar steps to those in the proof of~\cite[Lemma 4.3]{FJR19}. Let $f\in\C^A$ be a generic map $\CtC$, and consider any polynomial $P\in\C[w_1,w_2]$ satisfying 
\begin{itemize}

	\item $\Pi = \NP(P)$,

	\item the pair of curves $(\Vs(P),\Dsf)$ is Newton non-degenerate, and

	\item $\Vs(P)\cap \Dsf$ consists of transversal intersections.	

\end{itemize} Then, the lifted curve $\cL^*:=f^{-1}(\Vs(P))$ intersects $\Csf$ transversally in $\TT$, and thus
\begin{align}\label{eq:preimage-enter-crit}
\#\Vs(P)\cap \Dsf & = \#\cL^*\cap\Csf.
\end{align} Since $\# \Vs(P)\cap \Dsf=\MV(\Delta,~\Pi)$ according to Theorem~\ref{thm:BKK}, it is enough to show 
\begin{align}\label{eq:lft=mv}
\cL^*\cap\Csf & = \MV(\Sigma,~\Lft(\Pi)) - \cT_A(\Pi).
\end{align} 
Choosing $P$ generic enough, and making the substitution $(w_1,w_2)\leftrightarrow (f_1(z), f_2(z))$ in $P$, we get 
\[
\NP(\cL^*)=\Lft(\Pi),
\] where $P$ is expressed as $\sum c_b~w^b$. Then, results in~\cite{rojas1999toric} (see also~\cite[Section 5.5, pp. 121]{Ful93}) show that $\MV(\Sigma,~\Lft(\Pi))$ equals the total number of isolated intersection points (counted with multiplicities) of the two compactifications $\ovL:=\varsigma\circ \Phi_\Theta(\cL^*)$ and $\ovC:=\varsigma\circ \Phi_\Theta(\Csf)$ in the toric variety $X_\Theta$ corresponding to the polytope $\Theta:=\Sigma\oplus\Lft(\Pi)$ (see~\S\ref{sss:toric-lattice}). The above-referenced result can be viewed as a more detailed version of Bernstein's Theorem~\cite{Ber75}. Therefore, since $\#\ovC\cap\ovL = \#\partial C\cap\partial L + \#\cL^*\cap\Csf$, it is enough to show the following statement: 
\begin{equation}\label{eq:statement_2}
\text{\textit{$\cT_A(\Pi)$ is equal to the weighted number of intersection points of $\overline{\cL}\cap\ovC$ at the $1$-orbits of $X_\Theta$.}}
\end{equation}
Note that edges $\eta:=(\sigma,~\ell)$ of the pair of polytopes $(\Sigma,~\Lft(\Pi))$ are in bijection with edges of $\Sigma+\Lft(\Pi)$. Then, Theorem~\ref{thm:toric_2} shows that above edges $\eta$ correspond to $1$-orbits $V(\eta)$ of $X_\Theta$. Furthermore, if $\overline{\cL}\cap\ovC$ has a solution at some $V(\eta)$, then 
\begin{equation}\label{eq:Jacobian-lifted_restr}
\det\Jac_{z}(f)_\sigma = P(f(z))_\ell = 0
\end{equation} has a solution $z\in\TT$. It follows from the definitions that the restricted polynomial $P(f(z))_\ell$ can be expressed as
\begin{equation}\label{eq:expression_lift-restr}
P(f(z))_\ell= \sum_{b\in e\cap\N^2} c_b~f_{\gamma_1}^{b_1}~ f_{\gamma_2}^{b_2}, 
\end{equation} for some face $e\prec\Pi$, and a face $\gamma:=(\gamma_1,\gamma_2)\prec A^0$. One can choose the coefficients of $f_1$, $f_2$, and $P$ to be generic enough so that the system~\eqref{eq:Jacobian-lifted_restr} has a solution in $\TT$ if ad only if $\det\Jac_{z}(f)_\sigma$ has a factor of the form $f^\alpha_\gamma:=f_{\gamma_1}^{\alpha_1}~f_{\gamma_2}^{\alpha_2}$ for some $\alpha_1,\alpha_2\in\{0,1\}$. Using a case-by-case analysis on the types of edges of $A^0$, which we omit here, we deduce that if $\det\Jac_{z}(f)_\sigma$ is a factor of $f_{\gamma}^{\alpha}$, then $\det\Jac_{z}(f)_\sigma = \det\Jac_{z}(f_\gamma)$, and $\gamma$ is a long dicritical and not origin. 

To finish proving Statement~\eqref{eq:statement_2}, assume that $A^0$ has a long dicritical edge $\gamma$ that is not origin. Without loss of generality, we may furthermore set $\bm{0}\not\in\gamma_1$. Then, a simple computation shows that the following two equalities hold
\begin{align*}
\det\Jac_{z}(f)_\sigma & = f_{\gamma_1}\cdot h(z^v)\\
P(f(z))_\ell & = f_{\gamma_1}^{\beta}~\big(\sum_{b:=(\beta,k)\in\Ver(\Pi)} c_b f_{\gamma_2}^k\big),
\end{align*} where $v\in\Z^2$ is the spanning vector of $\gamma_2$, the polynomial $h$ is a univariate, the system $h(z^v)=P(f(z))_\ell = 0$ has no solutions in $\TT$, and $\beta:=\max \{b_1:(b_1,b_2)\in\Ver(\Pi)\}$. Note that each of $f_1$ and $\det\Jac_z f$ are Newton non-degenerate. Then, since $f_{\gamma_1}$ is quasi-homogeneous, it contributes to $\#~\gamma_1\cap\N^2 - 1$ solutions at $V(\eta)$, and thus $\partial L\cap\partial C$ consists of $\#~\gamma_1\cap\N^2 - 1$ points at $V(\eta)$, each counted with multiplicity $\beta$. Since $A^0$ has at most two dicritical long non-semi-origin edges $\gamma,\delta\prec A^0$, at most two orbits of $X_\Theta$ contribute to solutions at the boundary of $X_\Theta$. Summing up the number of those boundary solutions, counted with multiplicities, we get exactly $\cT_A(\Pi)$.
\end{proof}

Theorem~\ref{thm:correction_term-discriminant} extends to rational polytopes.

\begin{corollary}\label{cor:correction_term-discriminant}
Let $A\in\ccCgeq$ be a dicritical pair, and let $\Sigma$ and $\Delta$ be the two polytopes, corresponding to $A$, appearing in Theorem~\ref{thm:discriminant}. Then, for any \emph{rational} polytope $\Pi\subset\R_{\geq 0}^2$, the following equality holds
\begin{equation}\label{eq:correction_term-discriminant2}
\MV(\Delta,~\Pi) = \MV(\Sigma,~\Lft(\Pi)) - \cT_A(\Pi).
\end{equation} 
\end{corollary}

\begin{proof}
Let $\lambda\in \Q_{>0}$ be any value such that the scaled polytope $\lambda\cdot \Pi\subset \R^2$ is a lattice polytope. Since $\Lft(\lambda\cdot \Pi)$ will also be a lattice polytope, Theorem~\ref{thm:correction_term-discriminant} shows that
\begin{equation}\label{eq:correction_term-discriminant_l}
\MV(\Delta,~\lambda\cdot  \Pi) = \MV(\Delta,~\Lft(\lambda\cdot \Pi)) - \cT_A(\lambda\cdot  \Pi).
\end{equation} On the one hand, as scaling subsets in $\R^2$ commutes with the operations of taking the union, Minkowski sums, and convex hulls, we get 
\begin{equation}\label{eq:scaling_lifting}
\Lft(\lambda\cdot  \Pi) = \lambda\cdot\Lft(\Pi).
\end{equation} Similarly, from the definition of the correction term, we get
\[
\cT_A(\lambda\cdot \Pi) = \lambda~\cT_A(\Pi).
\] On the other hand, since the Mixed volume commutes with scaling of polytopes, we get $\MV(\Delta,~\lambda\cdot \Pi) = \lambda~\MV(\Delta,~\Pi)$, and $\MV(\Sigma,~\lambda\cdot \Lft(\Pi)) = \lambda~\MV(\Sigma,~ \Lft(\Pi))$. We obtain~\eqref{eq:correction_term-discriminant2} by combining the above identities to form the equality
\[
\lambda~\MV(\Delta,~\Pi) = \lambda~\MV(\Sigma,~\Lft(\Pi)) - \lambda~\cT_A(\Pi).
\] 
\end{proof}

\subsubsection{Correction term from mixed volumes}\label{sss:lattice-length} 
Let $P,\Pi\subset\R^2$ be full dimensional polytopes having the same inner normal fan $\cF$, and let $\rho_1, \dots, \rho_h$ be the rays in $\cF$ (i.e. its one-dimensional cones). For any $i = 1, \dots, h$ we set $\Pi_i\subset \Pi$ to be the polytope below (see the red polytope in Figure~\ref{fig:difference-volumes}):
\[  
\Pi_i := \{ q \in \R^2 \mid ~\langle q, \rho_i \rangle \geq \langle r, \rho_i \rangle + 1,~\forall r \in \Pi\}.
\]

\begin{lemma}
\label{lem:2-d-difference}
    The difference of volumes $\Vol(\Pi) - \Vol(\Pi_i)$ is equal 
    to the lattice length $\#~\Pi ^{\rho_i} \cap \Z^2 -1$ plus an error term
    which only depends on the fan.
\end{lemma}
\begin{proof}
We display the difference of $\Vol(\Pi)$ and $\Vol(\Pi_i)$ in Figure~\ref{fig:difference-volumes}. On the right, it is split into a yellow part of volume equal to the lattice length $\#~\Pi ^{\rho_i} \cap \Z^2 -1$, and a green part.
The green part constitutes a (possibly negative) error term equal to 
$ \frac{1}{ 2 \| \rho_i \|   }  (\operatorname{tan}(\alpha) +\operatorname{tan}(\beta)$.
Here $\alpha$ and $\beta$ denote the angles
$\alpha = \sphericalangle(\rho_1, \rho_2) + \frac{\pi}{2}$, 
$\beta = \sphericalangle(\rho_2, \rho_3) - \frac{\pi}{2}$.
In particular, the error term
depends only on the vectors $\rho_1, \dots, \rho_h$.
\end{proof}

\begin{figure}[h]
\centering
\includegraphics[width = 10cm]{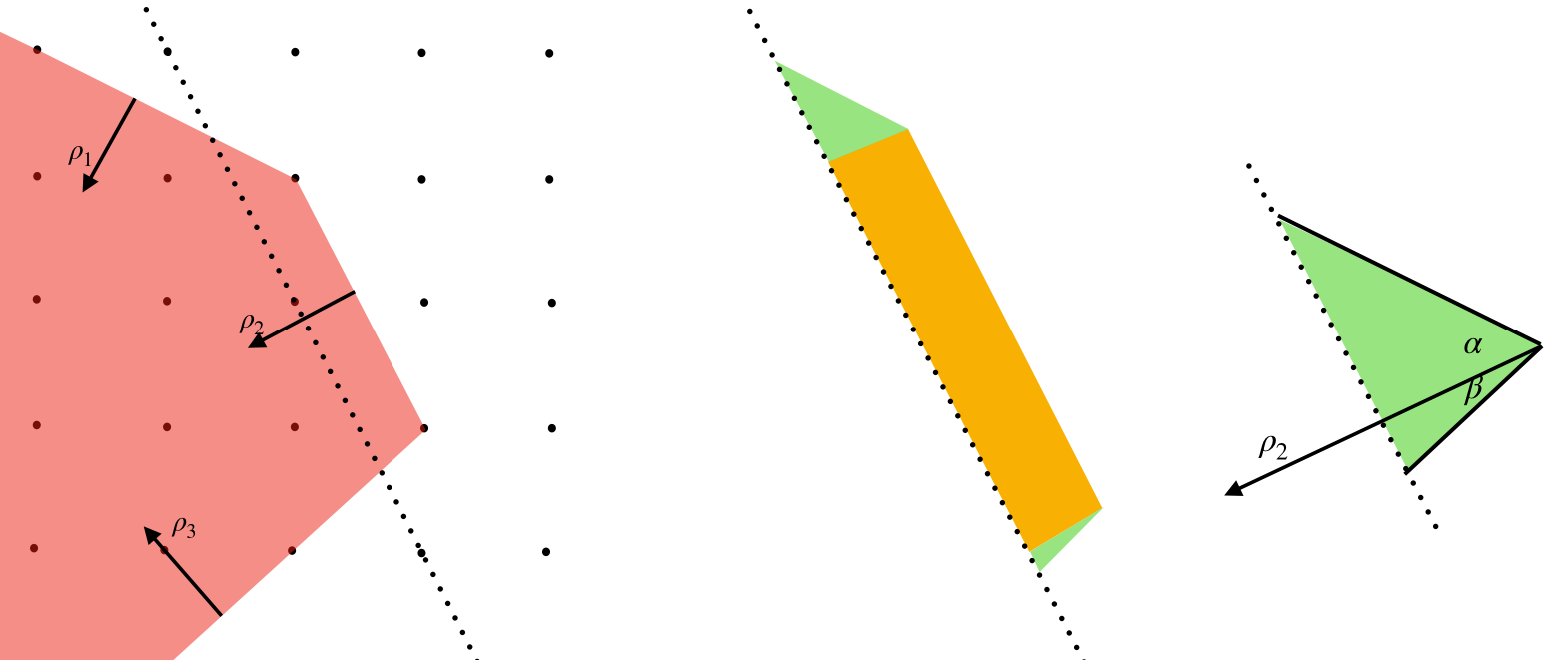} \vspace{-0.14in}
\caption{The polytope $\Pi_i$ is obtained from $\Pi$ by removing a strip of lattice width one.} \label{fig:difference-volumes}
\end{figure}

\begin{lemma}\label{lem:l-length-faces}
The lattice length of the face $P^{\rho_i}$ is equal to
    the difference of mixed volumes
    \[
    \# ~P^{\rho_i} \cap \Z^2 -1 = \MV(P, \Pi) - \MV(P, \Pi_i).
    \]
\end{lemma}
\begin{proof}
 Using the properties of the mixed volume, we obtain the expansion
\begin{equation}
\label{eq: equation to show in prop}
    \MV(P, \Pi) - \MV(P, \Pi_i) =  \Vol(P\oplus \Pi) - \Vol(P\oplus \Pi_i) - \left( \Vol(\Pi) - \Vol(\Pi_i) \right).
\end{equation}
    %Since $ \left(P+\Pi\right)_i = P+\Pi_i$, 
    Applying Lemma~\ref{lem:2-d-difference} to the right side of \eqref{eq: equation to show in prop} yields
\begin{equation}
\label{eq: small equation in proof}
        \#~\left( P\oplus \Pi  \right)^{\rho_i} \cap \Z^2 - 1 - \left( \# ~ \Pi ^{\rho_i} \cap \Z^2 -1 \right).
\end{equation}
    Lattice lengths of faces are additive in Minkowski sums, so we have
        \[
    \#~\left( P\oplus \Pi  \right)^{\rho_i} \cap \Z^2 - 1 =
    \#~P ^{\rho_i} \cap \Z^2 - 1
    +  \#~\Pi ^{\rho_i} \cap \Z^2 -1.
    \]
    Substituting this into ~\eqref{eq: small equation in proof}, and~\eqref{eq: small equation in proof} into \eqref{eq: equation to show in prop}
    finishes the proof.
\end{proof}

\begin{remark}
    Let
    $\rho_1, \dots, \rho_h$ be primitive, pairwise distinct, vectors in the plane, ordered clockwise,
    where $\rho_i = ( \rho_{i,1},\rho_{i,2} )$. We denote by $\cF$ the planar fan whose rays are generated by $\rho_1, \dots, \rho_h$.
    One can construct a polytope $\Pi$ with normal fan $\cF$ as follows.
    We inductively define the vertices $v_1, \dots, v_h$ of $\Pi$ by setting $v_1 = (0,0)$
    and $v_{i+1}:=v_i+\kappa_i ( - \rho_{i,2},\rho_{i,1} )$.
    Here $(\kappa_1, \dots, \kappa_h)$ is a tuple of natural numbers
    satisfying $\kappa_1 \rho_1+ \cdots+ \kappa_h\rho_h = 0$.
    Computing $\Pi$, such that the total $\kappa_1 + \cdots+ \kappa_h$ of the lattice length of its faces is minmized, can in practice be done by solving an integer linear program.
\end{remark}

We finish this part with the following Lemma whose proof can be found in the proof of Lemma~\ref{lem:Newton_polytopes-invar}.

\begin{lemma}\label{lem:Jacobian-polytope}
    Let $A\in\ccCgeq$ be a conical pair. Then, the corresponding polytope $\Sigma$ appearing in Theorem~\ref{thm:discriminant} is equal to
\[
\Sigma=
\{ v_1+v_2 \mid (v_1, v_2)\in A_1\times A_2 :\quad 
\operatorname{det}( v_1  v_2
) \neq 0\}
\]
\end{lemma}

\subsubsection{The algorithm for the discriminant}\label{sss:the-alg-discr}
Lemma~\ref{lem:Jacobian-polytope}, together with Remark~\ref{rem:rays} give rise to Algorithm~\ref{euclid} for computing the normal fan of the corresponding polytope $\Delta$ from Theorem~\ref{thm:discriminant}; it takes as input a conical pair $A$ and returns the rays in the normal fan of the discriminant $\Delta$.

\begin{algorithm}
\caption{Computing the normal fan of the discriminant}\label{euclid}
\begin{algorithmic}[1]
\Procedure{getRayList($A_1,A_2$)}{}
\State $\Sigma \gets  getJacobian(A_1, A_2)$
\State $rayList \gets getRaysInOuterNormalFan(\Sigma)$
\State $resultingRays \gets \emptyset$
\For{ $\tau \in rayList $}
\State $\rho_1 \gets \max_{r \in A_1}( \langle r, \tau \rangle)$
\State $\rho_2 \gets \max_{s \in A_2}( \langle s, \tau \rangle)$
%\State $(r_1, r_2 \gets )$
%\State $(s_1, s_2 \gets)$
%\State$
%    \rho \gets
%     (r_1\tau_1 + r_2\tau_2, s_1\tau_1 +s_2\tau_2),$
\If{  $\rho_1 \neq 0 \text{ or } \rho_2 \neq 0$   }
\State $resultingRays \gets  resultingRays \cup \{(\rho_1, \rho_2)\}$
\EndIf
\EndFor\\
\Return $resultingRays$
\EndProcedure
\end{algorithmic}
\end{algorithm}

\begin{algorithm}
\caption{Computing lattice lengths of the discriminant}\label{euclid}
\begin{algorithmic}[1]
\Procedure{getLatticeLength($A_1,A_2,\Sigma,rayList,{\rho_{i}}$)}{}
\State $\Pi \gets PolytopeWithNormalFan(rayList)$
\State $ H \gets \{ m \in \R^2 \mid \ \forall \pi \in \Pi : \langle m, \rho_i \rangle \geq \langle \pi , \rho_i \rangle  + 1 \} $
\State $\Pi_i \gets \Pi \cap H$
\State $\cT_A(\Pi) \gets correction\_term(A_1, A_2,\Pi)$
\State $\cT_A(\Pi_i) \gets correction\_term(A_1, A_2,\Pi_i)$
\State $\Lft(\Pi) \gets getLiftedPolytope(\Pi)$
\State $\Lft(\Pi_i) \gets getLiftedPolytope(\Pi_i)$
\State $ vol\_1 \gets \MV(\Sigma,~\Lft(\Pi)) - \cT_A(\Pi)$
\State $ vol\_2 \gets \MV(\Sigma,~\Lft(\Pi_i)) - \cT_A(\Pi_i)$\\
\Return $vol\_1 - vol\_2$
\EndProcedure
\end{algorithmic}
\end{algorithm}

Lemma~\ref{lem:l-length-faces} allows us to uniquely reconstruct a planar polytope $P$ with fixed normal fan, up to translation, by evaluating mixed volumes $ \MV(P, \Pi)$. 
In particular, based on Theorem \ref{thm:correction_term-discriminant}, we obtain the following algorithm to compute $\Delta$.
The input consists of a conical pair $(A_1, A_2)$, the Newton polytope $\Sigma$ of the Jacobian, and a list of the rays $\rho_1, \dots, \rho_h$ of the normal fan of $\Delta$ with one distinguished element $\rho_i$. The output is the lattice length of the face $\Delta^{\rho_{i}}$.

\subsection*{Acknowledgments}  
The first author would like to thank Khazhgali Kozhasov for his valuable remarks on an earlier version of this paper.

\subsection*{Contact}
\ \\
  Boulos El Hilany,\\
TU Braunschweig, Institut f\"ur Analysis und Algebra, \\
 Universit\"atsplatz 2. 38106 Braunschweig, Germany\\ 
 \href{mailto:b.el-hilany@tu-braunschweig.de}{b.el-hilany@tu-braunschweig.de},\\
 \href{https://boulos-elhilany.com}{boulos-elhilany.com}.\\
 
 \noindent Kemal Rose,\\
  Max Planck Institute for Mathematics in the Sciences,\\
   Inselstraße 22, 04103 Leipzig, Germany\\
    \href{mailto:krose@mis.mpg.de}{krose@mis.mpg.de}.
    
\bibliographystyle{abbrv}					   % For the style

\begin{thebibliography}{10}

\bibitem{aoki1980topological}
K.~Aoki and H.~Noguchi.
\newblock On topological types of polynomial map germs of plane to plane.
\newblock {\em Memoirs of the School of Science and Engineering, Waseda
  University}, 44:133--156, 1980.

\bibitem{Ber75}
D.~N. Bernstein.
\newblock The number of roots of a system of equations.
\newblock {\em Funkcional. Anal. i Prilo\v zen.}, 9(3):1--4, 1975.

\bibitem{boardman1967singularities}
J.~Boardman et~al.
\newblock Singularities of differentiable maps.
\newblock {\em Publications Math{\'e}matiques de l'IH{\'E}S}, 33:21--57, 1967.

\bibitem{bochnak2013real}
J.~Bochnak, M.~Coste, and M.-F. Roy.
\newblock {\em Real algebraic geometry}, volume~36.
\newblock Springer Science \& Business Media, 2013.

\bibitem{BOUSQUETMELOU19961}
M.~Bousquet-Mélou.
\newblock A method for the enumeration of various classes of column-convex
  polygons.
\newblock {\em Discrete Mathematics}, 154(1):1--25, 1996.

\bibitem{cox2011toric}
D.~A. Cox, J.~B. Little, and H.~K. Schenck.
\newblock {\em Toric varieties}, volume 124.
\newblock American Mathematical Soc., 2011.

\bibitem{EHT21}
B.~El~Hilany and E.~Tsigaridas.
\newblock Computing the non-properness set of real polynomial maps in the
  plane.
\newblock {\em Vietnam J. Math.}, 2023.

\bibitem{esterov2013discriminant}
A.~Esterov.
\newblock The discriminant of a system of equations.
\newblock {\em Advances in Mathematics}, 245:534--572, 2013.

\bibitem{FJ17}
M.~Farnik and Z.~Jelonek.
\newblock On quadratic polynomial mappings of the plane.
\newblock {\em Linear Algebra Appl.}, 529:441--456, 2017.
%%%%%

\bibitem{FJM18}
M.~Farnik, Z.~Jelonek, and P.~Migus.
\newblock On quadratic polynomial mappings from the plane into the {$n$}
  dimensional space.
\newblock {\em Linear Algebra Appl.}, 554:249--274, 2018.

\bibitem{FJR19}
M.~Farnik, Z.~Jelonek, and M.~A.~S. Ruas.
\newblock Whitney theorem for complex polynomial mappings.
\newblock {\em Math. Z.}, Aug 2019.

\bibitem{Fuk76}
T.~Fukuda.
\newblock Types topologiques des polyn\^{o}mes.
\newblock {\em Inst. Hautes \'{E}tudes Sci. Publ. Math.}, (46):87--106, 1976.

\bibitem{Ful93}
W.~Fulton.
\newblock {\em Introduction to toric varieties}.
\newblock Princeton University Press, 1993.

\bibitem{GKZ08}
I.~M. Gelfand, M.~Kapranov, and A.~Zelevinsky.
\newblock {\em Discriminants, resultants, and multidimensional determinants}.
\newblock Springer Science \& Business Media, 2008.

\bibitem{golubitsky2012stable}
M.~Golubitsky and V.~Guillemin.
\newblock {\em Stable mappings and their singularities}, volume~14.
\newblock Springer Science \& Business Media, 2012.
%%%%%
\bibitem{Jel93}
Z.~Jelonek.
\newblock The set of points at which a polynomial map is not proper.
\newblock In {\em Annales Polonici Mathematici}, volume~58, pages 259--266,
  1993.

\bibitem{jelonek2016semi}
Z.~Jelonek.
\newblock On semi-equivalence of generically-finite polynomial mappings.
\newblock {\em Math. Z.}, 283(1):133--142, 2016.

\bibitem{jelonek2005quantitative}
Z.~Jelonek and K.~Kurdyka.
\newblock Quantitative generalized bertini-sard theorem for smooth affine
  varieties.
\newblock {\em Discrete \& Computational Geometry}, 34:659--678, 2005.

\bibitem{khovanskii1978newton}
A.~G. Khovanskii.
\newblock Newton polyhedra and the genus of complete intersections.
\newblock {\em Functional Analysis and its applications}, 12(1):38--46, 1978.

\bibitem{Kho16}
A.~G. Khovanski\u{\i}.
\newblock {N}ewton polytopes and irreducible components of complete
  intersections.
\newblock {\em Izv. Math.}, 80(1):263, 2016.

\bibitem{Kou76}
A.~G. Kouchnirenko.
\newblock Poly{\`e}dres de {N}ewton et nombres de {M}ilnor.
\newblock {\em Invent. Math.}, 32(1):1--31, 1976.

\bibitem{lando2004graphs}
S.~K. Lando, A.~K. Zvonkin, and D.~B. Zagier.
\newblock {\em Graphs on surfaces and their applications}, volume 141.
\newblock Springer, 2004.

\bibitem{levine1965elimination}
H.~I. Levine.
\newblock Elimination of cusps.
\newblock {\em Topology}, 3:263--296, 1965.

\bibitem{looijenga1974complement}
E.~Looijenga.
\newblock The complement of the bifurcation variety of a simple singularity.
\newblock {\em Inventiones mathematicae}, 23(2):105--116, 1974.

\bibitem{lyashko1983geometry}
O.~Lyashko.
\newblock Geometry of bifurcation diagrams.
\newblock {\em Itogi Nauki i Tekhniki. Seriya" Sovremennye Problemy Matematiki.
  Noveishie Dostizheniya"}, 22:94--129, 1983.

\bibitem{mather1973generic}
J.~N. Mather.
\newblock Generic projections.
\newblock {\em Ann. of Math.}, 98(2):226--245, 1973.

\bibitem{mather1973thom}
J.~N. Mather.
\newblock On thom--boardman singularities.
\newblock In {\em Dynamical systems}, pages 233--248. Elsevier, 1973.

\bibitem{mather1973stratifications}
J.~N. Mather.
\newblock Stratifications and mappings.
\newblock In {\em Dynamical systems}, pages 195--232. Elsevier, 1973.

\bibitem{Nak84}
I.~Nakai.
\newblock On topological types of polynomial mappings.
\newblock {\em Topology}, 23(1):45--66, 1984.

\bibitem{rojas1999toric}
J.~M. Rojas.
\newblock Toric intersection theory for affine root counting.
\newblock {\em J. Pure Appl. Algebra}, 136(1):67--100, 1999.

\bibitem{Sab83}
C.~Sabbah.
\newblock Le type topologique \'{e}clat\'{e} d'une application analytique.
\newblock In {\em Singularities, {P}art 2 ({A}rcata, {C}alif., 1981)},
  volume~40 of {\em Proc. Sympos. Pure Math.}, pages 433--440. Amer. Math.
  Soc., Providence, RI, 1983.

\bibitem{Saeki89}
O.~Saeki.
\newblock Topological types of complex isolated hypersurface singularities.
\newblock {\em Kodai mathematical journal}, 12(1):23--29, 1989.

\bibitem{thom1962stabilite}
R.~Thom.
\newblock La stabilit{\'e} topologique des applications polynomiales.
\newblock {\em Enseign. Math.(2)}, 8:24--33, 1962.

\bibitem{Tho69}
R.~Thom.
\newblock Ensembles et morphismes stratifi{\'e}s.
\newblock {\em Bull. Am. Math. Soc.}, 75(2):240--284, 1969.

\bibitem{whitney1955singularities}
H.~Whitney.
\newblock On singularities of mappings of {E}uclidean spaces {I}. {M}appings of
  the plane into the plane.
\newblock {\em Annals of Mathematics}, 62(3):374--410, 1955.

\end{thebibliography}

\def\cprime{$'$}

\end{document}